\documentclass[pdftex,12pt,a4paper, reqno]{amsart}
\usepackage[pdftex]{graphicx}
\DeclareGraphicsExtensions{.png,.pdf,.svg}
\usepackage{thumbpdf}
\usepackage{comment}
\usepackage{subcaption}
\graphicspath{{Figures/}}
\usepackage[english]{babel}
\usepackage[utf8]{inputenc}
\usepackage{amssymb}
\usepackage{amsfonts}
\usepackage{amsmath}
\numberwithin{equation}{section}
\usepackage{amsthm}
\usepackage{color}
\usepackage{url}
\usepackage[shortlabels]{enumitem}
\usepackage{graphicx}
\usepackage{scalerel}
\usepackage{stackengine}
\usepackage{wasysym}
\usepackage{mathtools}
\usepackage{url}
\usepackage{tgtermes}
\usepackage{xpatch}
\usepackage{amsaddr}
\usepackage{marginnote}
\usepackage{bm}       

\usepackage{titlesec}

\titleformat{\section}
  {\normalfont\Large\bfseries}{\thesection}{1em}{}
\titleformat{\subsection}
  {\normalfont\large\bfseries}{\thesubsection}{1em}{}
\titleformat{\subsubsection}
  {\normalfont\normalsize\bfseries}{\thesubsubsection}{1em}{}
\makeatletter
\renewcommand{\@secnumfont}{\bfseries\normalsize}
\xpatchcmd{\@thm}{.}{}{}{}
\makeatother
\calclayout

\def\calA{{\mathcal A}}
\def\calB{{\mathcal B}}

\def\calE{{\mathcal E}}

\def\calI{{\mathcal I}}
\def\calK{{\mathcal K}}
\def\calL{{\mathcal L}}
\def\calM{{\mathcal M}}

\def\calP{{\mathcal P}}

\def\calV{{\mathcal V}}

\newcommand\sub[2]{{#1}_{\scriptscriptstyle{#2}}}

\newcommand{\R}{{\mathbb R}}
\newcommand{\Q}{{\mathbb Q}}
\newcommand{\C}{{\mathbb C}}
\newcommand{\T}{{\mathbb T}}
\newcommand{\Z}{{\mathbb Z}}
\newcommand{\N}{{\mathbb N}}

\newcommand{\Ni}[1]{\sub{N}{{\rm F},#1}} 

\newcommand{\INF}{\mathcal{I}_{\NF}}
\newcommand{\nuNi}[1]{\nu_{#1}} 
\newcommand{\muNi}[1]{\mu_{#1}} 

\def\Ntot{{\sub{N}{{\rm D}}}}
\def\NF{{\sub{N}{{\rm F}}}}
\def\ee{\mathrm{e}}
\def\ii{\mathrm{i}}
\def\hatrho{\hat{\rho}}

\def\st{\ |\ } 
\newcommand{\ep}{\varepsilon}

\newcommand{\norm}[1]{{|{#1}|}}  
\newcommand{\NORM}[1]{{\|{#1}\|}} 
\def\Dif{{\rm D}}
\def\dif{{\rm d}}

 \newcommand\reallywidetilde[1]{\ThisStyle{%
  \setbox0=\hbox{$\SavedStyle#1$}%
  \stackengine{-.1\LMpt}{$\SavedStyle#1$}{%
    \stretchto{\scaleto{\SavedStyle\mkern.2mu\AC}{.5150\wd0}}{.6\ht0}%
  }{O}{c}{F}{T}{S}%
}}

\def\re{\mathrm{Re}}
\def\im{\mathrm{Im}}

\def\Id{\mathrm{Id}}

\newtheorem{definition}{Definition}[section]
\newtheorem{theorem}[definition]{Theorem}
\newtheorem{lemma}[definition]{Lemma}

\newtheorem{corollary}[definition]{Corollary}

\newtheorem{remark}[definition]{Remark}

\newcommand{\thmref}[1]{Theorem~\ref{#1}}
\newcommand{\lemref}[1]{Lemma~\ref{#1}}
\newcommand{\corref}[1]{Corollary~\ref{#1}}

\newcommand{\secref}[1]{Section~\ref{#1}}
\newcommand{\remref}[1]{Remark~\ref{#1}}

\renewenvironment{abstract}
 {\small\quotation\noindent\textbf{\abstractname}\quad}
 {\endquotation}

\begin{document}

\title[Computation of invariant tori]{
Computer-Assisted Proofs of Existence of \\ Fiberwise Hyperbolic Invariant Tori in \\ Quasi-periodic Systems via Fourier Methods}

\author[Alex Haro \and Eric Sandin Vidal]{Alex Haro$^{\mbox{\scriptsize *1,2}}$ \and Eric Sandin Vidal$^{\mbox{\scriptsize 3}}$}
\address{$^{\mbox{\scriptsize 1}}$ \small Departament de Matem\`atiques i Inform\`atica, Universitat de Barcelona, Gran Via 585, 08007 Barcelona, Spain}
\address{$^{\mbox{\scriptsize 2}}$ \small Centre de Recerca Matem\`atica, Edifici C, Campus Bellaterra, 08193 Bellaterra, Spain}
\address{$^{\mbox{\scriptsize 3}}$ \small Department of Mathematics, Vrije Universiteit Amsterdam - Faculty of Science, De Boelelaan 1111, 1081 HV Amsterdam, The Netherlands}
\email{alex@maia.ub.es, e.sandin.vidal@vu.nl}

\thanks{%
    \hspace*{-\parindent}\textbf{Keywords} invariant tori, quasi-periodically forced system, normal hyperbolicity, Fourier methods, computer-assisted proof
    A.H. has been supported by the Spanish grant PID2021-125535NB-I00 (MCIU/AEI/FEDER, UE), and by the Spanish State Research Agency, through the Severo Ochoa and Mar\'ia de Maeztu Program for Centers and Units of Excellence in R\&D (CEX2020-001084-M)\\
    *Corresponding author: Alex Haro
}

\date{}

\maketitle

\begin{abstract}
The goal of this paper is to provide a methodology to prove the existence of (fiberwise hyperbolic) real-analytic invariant tori in real-analytic quasi-periodic skew-product dynamical systems that present nearly-invariant tori of the same characteristics. The methodology was based on the application of a Newton-Kantorovich theorem whose hypotheses were tested using Fourier analysis methods for a numerical approximation of the parameterization of an invariant torus.
\end{abstract}

\section{Introduction}

Dynamical systems subject to a quasi-periodic forcing are of great significance beyond mere theoretical interest. Many works in physics, engineering, and other areas of science greatly depend on a deep understanding of systems undergoing such motions either in a direct or indirect way (see \cite{agrawal2020universality,
feudel1997phase, prasad2001strange} and the references therein, to name a few). This quasi-periodic motion forced upon the system acts as a perturbation, leading us to the natural question of existence of invariant objects robust to such perturbations, a common practice in dynamical systems. As it is well-known, this persistence relates to the property of normal hyperbolicity \cite{mane1978persistent}.

In this context, models are described as skew-product systems (or bundle maps) over a torus in which the dynamics is given by an ergodic rotation, and the response quasi-periodic solutions are geometrically described as invariant tori that are graphs over the torus. What we present here is a new methodology that takes advantage of the quasi-periodic nature of the setting by the use of the Fourier methods in \cite{FiguerasHaroLuque2017}. For such purpose, we will consider the simplest setting, in which the torus is {\em fiberwise} hyperbolic, since hyperbolicity allows for a seamless application of a Newton-Kantorovich argument, allowing us to discard more complicated methods such as KAM or Nash-Moser (as in \cite{FiguerasHaroLuque2017,FiguerasHaroLuque2020-Rotations}).

In many examples in the literature, such systems are real-analytic (see the references above). A natural question is, then, if the
response tori to the quasi-periodic forcing inherits such a regularity. Notice that it is well-known that {\em normally} hyperbolic invariant tori are, in general, only finitely differentiable even if the system is infinitely differentiable or analytic, with the regularity being related with the rates of contraction/expansion of the stable/unstable bundles compared with such rates on the tangent bundle \cite{fenichel1974asymptotic, fenichel1977asymptotic, hirschPS}. However, thanks to the special skew-product form of the quasi-periodic systems, and the fact that the internal motion on the torus is a rotation, the fiberwise hyperbolic invariant tori are as smooth as the systems \cite{harodelallave2006parameterization}.

With this context in mind, we turn the focus of this work toward real-analytic skew-product dynamical systems with quasi-periodic forcing, for which we are interested in proving the existence of invariant tori close to approximate invariant tori (produced, possibly, by numerical methods). Following the methodology of a validation theorem as in \cite{FiguerasHaro2012, DelallaveHaro2006} (or also called a radii polynomial approach, see \cite{van2018continuation, lessard2017computer}), applied over real-analytic objects, we will use  Fourier analysis to rigorously verify the hypotheses of a Newton-Kantorovich theorem. The Fourier methods in \cite{FiguerasHaroLuque2017} avoid standard convolution techniques and can be applied to problems with non-polynomial non-linearities (see also \cite{FiguerasHaroLuque2020-Rotations,linroth2023computer-arxiv} for other applications).
For another approach, based on the so-called Fourier models (that include a truncated Fourier series and a bound of the tails), see \cite{FiguerasHaro2012}. With such goal in mind, computations will be implemented with interval arithmetics on a computer, which will then validate our object of interest. Such procedure is  called a computer-assisted proof.

Our approach in this study involves the application of several techniques, including functional equations, Newton-Kantorovich theorems, and Fourier series analysis, to effectively manage the numerical aspects inherent to our subject matter. This combination of methods allows us to approximate and manipulate our data in a manner that is both accurate and adaptable. Importantly, the flexibility of this methodology means that it can be easily tailored to address a wide range of problems sharing similar characteristics (such as the validation of periodic orbits or periodic solutions of ODE systems). By employing these tools, we not only address the specific challenges encountered in our research but also lay the groundwork for potential applications in various other problem-solving scenarios.

The work will be divided in several parts, starting with the setting of the problem, followed by the validation theorem, which will show how to prove the existence of an invariant torus, a section on Fourier analysis and how it will help in our endeavor, and finally, a section explaining how the computer-assisted proof itself is carried out  for fiberwise invariant tori in a quasi-periodically forced standard map \cite{FiguerasHaro2012}.

\section{Setting}
\label{sec:setting}

\subsection{Definitions}
Following standard practice, let $\T^d=(\R / \Z)^d$ be the real torus, and $\T_\C^d = \T^d + \ii\R^d$ be the complex torus. We denote a \textit{complex strip} (in $\T_\C^d$) of width $\rho>0$ by
\begin{equation*}
    \T_\rho^d = \{ \theta \in \T_\C^d \st \im |\theta_i|<\rho \, , \, i=1, \ldots , d \},
\end{equation*}
and by $\bar\T^d_\rho$ its closure. We denote by $C^0(\bar\T^d_\rho,\C^m)$ the Banach space of continuous
functions $u: \bar\T^d_\rho\to\C^m$ endowed with the norm \smallskip
\begin{equation*}
    \NORM{u}_\rho = \sup_{\theta \in \T_\rho^d} \norm{u(\theta)},
\end{equation*} 
where $\norm{\cdot}$ is the supremum norm in $\C^m$. We denote by $A(\bar\T^d_\rho,\C^m)$ the  Banach space of continuous functions $u:\bar\T_{\rho}^{d}\to\C^m$, holomorphic on $\T_{\rho}^{d}$ and such that $u(\T^{d})\subset\R^m$ (so $u$ is real-analytic), endowed with the supremum norm. Notice that $A(\bar\T^d_\rho,\C^m) \subset C^0(\bar\T^d_\rho,\C^m)$, and that the inclusion is closed.

 Let $\calA\subset \R^n \times \T^d$ be an annulus, that is, an open set homotopic to $\calV \times \T^d$, where $\calV\subset\R^n$ is open.
Let $\calB \subset \C^n \times \T^d_\C$ be a complex neighborhood of the annulus $\calA$. Such $\calA$ will be the  real domain of the skew-product dynamical system defined in the following section, and $\calB$ will be the domain of a complex extension of such dynamical system.

\subsection{Skew-Product Dynamical Systems}
In this paper, we consider real-analytic \textit{skew-product} dynamical systems
\begin{equation*}
\begin{array}{llcl}
\hat{F} = (F,f):& \multicolumn{1}{c}{\calB} & \longrightarrow & \C^n \times \T_\C^d\\
& (z, \theta) & \longrightarrow & (F(z, \theta), f(\theta))\\
\end{array}\ ,
\end{equation*}
where $\pi_{\T_\C^d} (\calB) = \T_r^d$ for a given $r>0$, and $F\colon \calB \to \C^n$ and $f \colon \T^d_r \to \T_\C^d$ are real-analytic, and thus $F(\calA) \subset \R^n$ and $f(\T^d)=\T^d$. More particularly, we will work with skew-products defined over rotations with angle $\omega$, that is, $f(\theta) = R_\omega (\theta) = \theta + \omega$. In such systems, if $\omega \in \R^d$ such that $\forall k \in \Z^d \backslash \{ 0 \}$ and $\forall p \in \Z$, $k \omega \neq p$, then $(F,R_\omega)$ is called a quasi-periodic skew-product. In any case, $R_\omega(\T^d_r) = \T^d_r$. As mentioned before, similar cases with lower regularity conditions have been extensively studied in \cite{FiguerasHaro2012, DelallaveHaro2006}.

The focus of this work are real-analytic tori, whose parameterization is given by $K \in A(\bar\T^d_\rho,\C^n)$ with $0< \rho < r$, which defines a real-analytic section of the bundle $\C^n \times \T^d_\C$, with graph $\calK_\rho= \{ (K(\theta),\theta) \st \theta\in \bar\T^d_\rho \}$. We often abuse notation and refer to $K$ as a torus, rather than a section or the parameterization of the torus $\calK_\rho$. If the map $K$ satisfies the functional equation
\begin{equation}\label{invarianceequation}
F(K(\theta-\omega), \theta - \omega) - K(\theta)  = 0,
\end{equation}
then the torus $\calK_\rho$ is invariant under $(F,R_\omega)$ and its inner dynamics is the rigid rotation by $\omega$.

The invariance equation \eqref{invarianceequation}
can be rewritten in functional terms. Let \[\calB^* = \{ K \in A(\bar\T^d_\rho,\C^n) \st \forall \theta \in \bar\T^d_\rho , (K(\theta),\theta) \in \calB \}.\]
Then, let $E \colon \calB^* \to A(\bar\T^d_\rho,\C^n)$ be the
operator defined as
\begin{equation*}
E(K)(\theta) = F(K(\theta-\omega), \theta-\omega) - K(\theta)\ ,
\end{equation*}
then $K$ is an invariant torus for $(F,R_\omega)$ if and only if
\begin{equation}\label{eq:T zero}
E(K)(\theta) = 0.
\end{equation}

The differential of \eqref{eq:T zero} is the linear operator
$\Dif E(K) : A(\bar\T^d_\rho,\C^n) \to A(\bar\T^d_\rho,\C^n)$ defined as
\begin{equation*}
\Dif E(K) \Delta(\theta) =
\Dif_z F(K(\theta-\omega), \theta-\omega)
\Delta(\theta-\omega) - \Delta(\theta)\ .
\end{equation*}
Notice that the nature of the solutions of \eqref{eq:T zero}, and hence of the bounded linear operator
$\Dif E(K)$, is
strongly related to the dynamical  properties of the
linearized dynamics around $K$.

\subsection{Transfer Operator and Hyperbolicity} \label{subsec:transfer}
The linear dynamics mentioned before is given by the vector bundle map
\begin{equation*}
\begin{array}{llcl}
(M, R_\omega): & \C^n \times \bar\T_\rho^d &
\longrightarrow & \C^n \times \bar\T_\rho^d\\
& (v, \theta) & \longrightarrow & (M(\theta) v, \theta+\omega)\\
\end{array}\ ,
\end{equation*}
where $M\colon \T_\rho^d \to \C^{n \times n}$
is the \textit{transfer matrix}
$M(\theta)= \Dif_z F(K(\theta), \theta)$.

To this linear quasi-periodic skew-product, we associate a functional object, the \textit{transfer operator} $\mathcal{M}$,
which is the bounded linear operator
$\calM: C^0(\bar\T^d_\rho,\C^n)
\to C^0(\bar\T^d_\rho,\C^n)$ defined as
\begin{equation*}
\mathcal{M}(\Delta)(\theta) = M(\theta-\omega)\Delta(\theta-\omega).
\end{equation*}
Notice that $A(\bar\T^d_\rho,\C^n)$ is invariant under the action of the transfer operator.

We say a real-analytic torus $K$ is fiberwise hyperbolic if the transfer operator $\calM$ is hyperbolic on $A(\bar\T^d_\rho,\C^n)$ (that is, its spectrum does not intersect the unit circle). This implies that $\calM - \calI$ is invertible in $A(\bar\T^d_\rho,\C^n)$.

\begin{remark}
We emphasize that if $\calM$ is hyperbolic as an operator in $C^0(\bar\T^d_\rho,\C^n)$, then it is hyperbolic as an operator in $A(\bar\T^d_\rho,\C^n)$. This follows from the analyticity of the spectral subbundles \cite{johnson1980analyticity} and the spectral inclusion \cite{canadellHaro2017}
    \begin{equation*}
        \text{Spec}(\calM , A(\bar\T^d_\rho,\C^n) ) \subset \text{Spec}(\calM , C^0(\bar\T^d_\rho,\C^n) ).
    \end{equation*}
\end{remark}

As  is well-known, hyperbolicity of $(M, R_\omega)$ relates to  the existence of stable and unstable bundles. In this work, we will assume these bundles are trivializable (something that sometimes can be obtained using double-covering tricks).

In particular, in the applications of this paper, we will consider a real-analytic torus $K$ to be fiberwise hyperbolic if there are continuous maps $P \colon \bar\T_\rho^d \to \C^{n \times n}$ and $\Lambda \colon \bar\T_\rho^d \to \C^{n \times n}$, real-analytic in $\T_\rho^d$, such that
\begin{equation*}
    P(\theta + \omega)^{-1} M(\theta) P(\theta) - \Lambda (\theta) = 0,
\end{equation*}
and
\[
\Lambda(\theta)=\begin{pmatrix}
\Lambda_s(\theta) & 0\\
0 & \Lambda_u(\theta)
\end{pmatrix}
\]
with each block $\Lambda_s(\theta) \in \C^{n_s \times n_s}$, $\Lambda_u(\theta) \in \C^{n_u \times n_u}$ satisfying $\NORM{\Lambda_s(\theta)}_\rho \leq \lambda_s < 1$ and $\NORM{(\Lambda_u(\theta))^{-1}}_\rho \leq \lambda_u < 1$ for some $\lambda_s$ and
$\lambda_u$, where the norms for matrices of real-analytic functions are naturally extended to
\begin{equation*}
\NORM{\Lambda}_{\rho} = \max_{i=1,\ldots,n} \sum_{j=1}^{n}
\NORM{\Lambda_{i,j}}_{\rho}.
\end{equation*}

\begin{remark}
    We have tailored the definition of hyperbolicity to the methods of this paper, especially regarding the reducibility to simple dynamics. In particular, the torus being hyperbolic means that under an appropriate linear change of variables (given by the map $P$), the transfer map is reducible into a block-diagonal form (given by the map $\Lambda$), with stable and unstable bundles. In such case, the first $n_s$ columns of $P$ represent the directions of the stable subbundles, and the remaining $n_u$ columns represent the directions of the unstable subbundles.
\end{remark}

\begin{remark}
From the previous constructions, one obtains the spectral gap condition
\begin{equation*}
        \text{Spec}(\calM , A(\bar\T^d_\rho,\C^n) ) \cap \{ z \in \C \st \lambda_s < |z| < \lambda_u \} = \emptyset ,
    \end{equation*}
which implies the hyperbolicity of the transfer operator in $A(\bar\T^d_\rho,\C^n)$.
\end{remark}

\begin{remark}
    In an abuse of notation, we will refer to the norm of a linear operator (like the transfer operator $\calM$) in $A(\bar\T^d_\rho,\C^n)$ as
    \begin{equation*}
        \NORM{\calM}_\rho = \NORM{\calM}_{A(\bar\T^d_\rho,\C^n)} = \NORM{M}_{\rho}.
    \end{equation*}
\end{remark}

Then, if $\calL_s$ and $\calL_u$ are the transfer operators associated with $\Lambda_s$ and $\Lambda_u$ respectively, we have
\begin{equation*}
    \NORM{(\calL_s - \calI)^{-1}}_\rho \leq \frac{1}{1-\lambda_s} \quad \text{and} \quad \NORM{(\calL_u - \calI)^{-1}}_\rho \leq \frac{\lambda_u}{1-\lambda_u} .
\end{equation*}

We emphasize that, since $\Dif E(K) = \calM - \calI$, the
hyperbolicity property of the transfer operator $\calM$ implies the
invertibility of $\calM - \calI$ and therefore of $\Dif E(K)$ and hence the applicability of Newton's method to equation \eqref{eq:T zero}. For more details on the relation between the transfer operator and the hyperbolicity of the linear dynamics, we refer the reader to \cite{canadellHaro2017, FiguerasHaro2012}.

\section{Validation Results}

In this section, we present a Newton-Kantorovich lemma, which shows existence of zeroes in differentiable maps between Banach spaces given approximate zeroes satisfying certain conditions. The main theorem will show afterward how to apply it to our case.

\subsection{A Newton-Kantorovich Lemma}

The following result is standard in the literature of computer-assisted proofs (see e.g.,  \cite{van2018continuation, FiguerasHaro2012, DelallaveHaro2006, lessard2017computer}).
Its proof is essentially the application of a contraction principle.

\begin{lemma}\label{NewtonKantorovich}
    Let $E : B_R(x_0) \subset X \to Y$ be a $C^1$ mapping where $X, Y$ are Banach spaces and $x_0 \in X$, such that $\Dif E(x_0) : X \to Y$ is bounded and invertible (with bounded inverse by the Banach isomorphism theorem). Let $\ep, \sigma$ be constants and $b : (0, R) \to \R_+$ a function such that
    \begin{enumerate}[a)]
        \item $\NORM{E(x_0)} \leq \ep$;\label{lem:cond:a}
        \item $\NORM{\Dif E (x_0)^{-1}} \leq \sigma$;\label{lem:cond:b}
        \item  $\forall r<R$, $\forall x \in \Bar{B}_r(x_0)$, $\NORM{\Dif E(x) - \Dif E(x_0)} \leq b(r) \NORM{x - x_0}$.\label{lem:cond:c}
    \end{enumerate}
    Assume also that for a certain $r_* <R$
    \begin{enumerate}[1)]
        \item $\frac{1}{2} \sigma b(r_*) r_*^2 - r_* + \sigma \ep \leq 0$;\label{cond:1}
        \item $ \sigma b(r_*) r_*<1$.\label{cond:2}
    \end{enumerate}
Then there exists a unique $x_* \in \Bar{B}_{r_*}(x_0)$ such that $E(x_*) = 0$. Moreover, $\Dif E(x_*)$ is bounded and invertible, and
\begin{equation*}
    \NORM{x_*-x_0} \leq \frac{\sigma \ep}{1-\sigma b(r_*) r_*} \ , \quad
    \NORM{\Dif E(x_*)^{-1}} \leq \frac{\sigma}{1 - \sigma b(r_*) \NORM{x_* - x_0}} \leq \frac{\sigma}{1 - \sigma b(r_*) r_*}.
\end{equation*}
\end{lemma}

\begin{proof}
We start by defining a map $T : B_R (x_0) \to X$ by $T(x) = x - \Dif E(x_0)^{-1} E(x)$ such that the zeroes of $E$ correspond to the fixed points of $T$ (the iterates of $T$ give a quasi-Newton method for the zeroes of $E$). The standard procedure is to apply Banach's fixed-point theorem. For that, we search for conditions on $r<R$ to apply the theorem.

We can start by checking where does $T$ map the ball $\bar{B}_{r}(x_0)$, so for any $x \in \Bar{B}_r(x_0)$,
    \begin{align*}
        \|&T(x) - x_0\| \\
        &= \NORM{x - x_0 - \Dif E(x_0)^{-1} E(x)} \\
        &= \NORM{x - x_0 - \Dif E(x_0)^{-1} ( E(x)-E(x_0) +E(x_0))} \\
        &= \|(x-x_0) - \Dif E(x_0)^{-1} \int_0^1 \Dif E(x_0 + t (x-x_0)) \dif t \ (x-x_0) - \Dif E(x_0)^{-1} E(x_0)\| \\
        &= \sigma \frac{1}{2} b(r) \NORM{x-x_0}^2 + \sigma \ep \leq \sigma \frac{1}{2} b(r) r^2 + \sigma \ep
    \end{align*}
    where the $b(r) \NORM{x-x_0}$ factor is given from \ref{lem:cond:c}. Next, we can compute the Lipschitz constant $L_r$ of $T$. Let $x_1, x_2 \in \Bar{B}_r(x_0)$,
\begin{align*}
        \|&T(x_2) - T(x_1)\| \\
        &= \NORM{x_2 - x_1 - \Dif E(x_0)^{-1} ( E(x_2) - E(x_1))} \\
        & = \| -\Dif E(x_0)^{-1} \int_0^1 \big(\Dif E(x_1 + t (x_2 - x_1)) - \Dif E(x_0)\big) (x_2-x_1) \dif t \| \\
        & \leq \sigma \int_0^1 b(r) \NORM{x_1 + t(x_2 - x_1) - x_0} \dif t \|x_2 - x_1\| \\
        &\leq \sigma b(r) r \ \NORM{x_2 - x_1}
    \end{align*}
which follows by hypothesis \ref{lem:cond:c}. So, the Lipschitz constant $L_r$ of $T$ in $\Bar{B}_r(x_0)$ is $L_r=\sigma b(r) r$. Now, by hypotheses \ref{cond:1} and \ref{cond:2}, there is a $r_*<R$ such that $\|T(x) - x_0\| \leq r_*$, and hence $T(\Bar{B}_{r_*}(x_0)) \subset \Bar{B}_{r_*}(x_0)$, and $L_{r_*} <1$, which implies that $T$ is a contraction and therefore $\exists ! \ x_* \in \Bar{B}_{r_*}(x_0)$ such that $T(x_*)=x_*$ (that is, $E(x_*)=0$). Furthermore, from the Banach contraction mapping theorem we also get
\begin{equation*}
    \NORM{x_*-x_0} \leq \frac{\NORM{T(x_0)-x_0}}{1-L_{r_*}} \leq \frac{\sigma \ep}{1-\sigma b(r_*) r_*}.
\end{equation*}
As for the last result, we see
\begin{align*}
    \Dif E(x_*) &= \Dif E(x_0) + \Dif E(x_*) - \Dif E(x_0)\\
    &= \Dif E(x_0) \big( \Id + \Dif E(x_0)^{-1} ( \Dif E(x_*) - \Dif E(x_0) ) \big)\ .
\end{align*}
By hypothesis, $\| \Dif E(x_0)^{-1} ( \Dif E(x_*) - \Dif E(x_0) ) \| \leq \sigma b(r_*) r_* <1$, which means that $\Id + \Dif E(x_0)^{-1} ( \Dif E(x_*) - \Dif E(x_0) )$ is invertible and the result follows.
\end{proof}

\begin{remark}
    In \cite{van2018continuation, lessard2017computer}, condition \ref{lem:cond:c} is slightly different, which makes \ref{cond:1} and \ref{cond:2} simplified in one condition
    \begin{equation}\label{rem:radii}
        \sigma b(r_*)r_*^2 -r_* +\sigma + \ep <0.
    \end{equation}
    Notice that, from hypotheses \ref{lem:cond:a}, \ref{lem:cond:b}, and \ref{lem:cond:c}, if \eqref{rem:radii} is satisfied, then \ref{cond:1} and \ref{cond:2} are also satisfied.
\end{remark}

\begin{remark}\label{rem::radiuses}
    Notice that we can apply \lemref{NewtonKantorovich} for any $r_*\in (r_-,r_+)$ where
    \begin{align*}
            r_- &= \inf \{ r \in (0, R) \st \tfrac{1}{2} \sigma b(r) r^2 - r + \sigma \ep \leq 0 \ \textit{and} \ \sigma b(r) r < 1\}, \\
            r_+ &= \sup \{ r \in (0, R)  \st \tfrac{1}{2} \sigma b(r) r^2 - r + \sigma \ep \leq 0 \ \textit{and} \ \sigma b(r) r < 1\} .
    \end{align*}
    Hence, there is a unique solution $x_*$ of $E(x)= 0$ in the ball $B_{r_+}(x_0)$, and moreover it is in $\bar B_{r_-}(x_0)$.
\end{remark}

\subsection{Validation Theorem}
This is the a posteriori theorem we will use to perform computer-assisted proofs of existence of fiberwise hyperbolic invariant tori.

\begin{theorem}\label{thm::validation}
Assume that given a $\rho>0$, we have a real-analytic torus $K_0: \bar{\T}_\rho^d \rightarrow \C^n$. Let $D_{\rho,R} = \{ (z, \theta) \st z \in \C^n \ , \ \theta \in \bar\T_\rho^d \ , \ \norm {z - K_0(\theta)} < R \}$ for a certain $R$ and $F : D_{\rho,R} \rightarrow \C^n$ be a real-analytic map, defining a skew-product over the rotation $\omega \in \R^d$. Assume that
\begin{enumerate}[a)]
    \item $\NORM{F(K_0(\theta - \omega), \theta - \omega) - K_0(\theta)}_\rho \leq \ep $;\label{condA}
    \item $M_0(\theta) = \Dif_z F(K_0(\theta), \theta)$, the corresponding linear skew-product, is hyperbolic in $A(\bar\T^d_\rho,\C^n)$ and $\NORM{(\calM_0 - \Id)^{-1}}_{A(\bar\T^d_\rho,\C^n)} \leq \sigma$;\label{condB}
    \item $\forall r <R$ and $\forall (z, \theta) \in \Bar{D}_{\rho,r}$, $\norm{\Dif_z F(z, \theta) - \Dif_z F(K_0(\theta), \theta)} \leq b(r) \norm{z-K_0(\theta)}$.\label{condC}
\end{enumerate}
If, moreover, for a certain $r_* <R$
    \begin{enumerate}[1)]
        \item $\frac{1}{2} \sigma b(r_*) r_*^2 - r_* + \sigma \ep \leq 0$;
        \item $ \sigma b(r_*) r_*<1$,
    \end{enumerate}
then there exists a unique $K_* \in \Bar{B}_{r_*}$ such that $E(K_*) = 0$. Moreover, the cocycle given by $M_*(\theta) = \Dif_z F(K_*(\theta), \theta)$ is hyperbolic and the torus is fiberwise hyperbolic.
\end{theorem}

\begin{proof}
    It is only necessary to adapt the conditions to match those of \lemref{NewtonKantorovich}. Define for $r \leq R$, $\Bar{B}_r (K_0) = \{ K \in A(\bar\T^d_\rho,\C^n) \st \NORM{K - K_0}_\rho \leq r \}$. Let $E:B_R(K_0) \subset A(\bar\T^d_\rho,\C^n) \to A(\bar\T^d_\rho,\C^n)$ be defined by $E(K_0)(\theta) := F(K_0(\theta - \omega), \theta - \omega) - K_0(\theta)$. Condition \ref{condA} is straightforward. For condition \ref{condB}, it is easy to see that for $K\in B_R(K_0)$,  $\Dif E(K) : A(\bar\T^d_\rho,\C^n) \to A(\bar\T^d_\rho,\C^n)$ is given by $\Dif E(K) \Delta(\theta) = M(\theta - \omega) \Delta(\theta - \omega) - \Delta(\theta)$. Thus, $\Dif E(K_0) = \calM_0 - \calI$. For condition \ref{condC}, we can just see that for any $K \in \Bar{B}_r$,
    \begin{align*}
        \|\Dif &E(K) - \Dif E(K_0)\|_{A(\bar\T^d_\rho,\C^n)} \\
        &=\NORM{\Dif_z F(K(\theta - \omega), \theta - \omega) - \Dif_z F(K_0(\theta - \omega), \theta - \omega)}_\rho \leq b(r) \NORM{K - K_0}_\rho .
    \end{align*}
    The result follows directly from applying \lemref{NewtonKantorovich}.
\end{proof}
Note that a similar theorem is proved in \cite{DelallaveHaro2006} in the $C^r$-category.

\begin{remark}
    Following \remref{rem::radiuses} and \thmref{thm::validation}, we can say that $\exists ! \ K_* \in B_{r_+}$ such that $ E(K_*) = 0 $ and $\NORM{K_* - K_0} \leq r_-$.
\end{remark}

\subsection{Hyperbolicity Control}

It is now natural to wonder how to derive the invertibility of $\Dif E(K_0) = \calM_0 - \calI$ from the hyperbolicity that we defined earlier. This amounts to calculating $(\calM_0 - \calI)^{-1}v(\theta)$, that is, solving $(\calM_0 - \calI)u(\theta) = v(\theta)$. For this purpose, it is useful to write the reducibility properties of our $M_0$ in terms of the transfer operator.

In this paper, we assume our approximate invariant torus $K_0$ to be approximately reducible in the sense we described in Section \ref{sec:setting}. In particular, we will see a method to produce bound \ref{condB} of Theorem \ref{thm::validation}.

\begin{lemma}
    Assume that there exists a map $P_1 \colon \T_\rho^d \to \C^{n \times n}$ whose approximate inverse is given by $P_2 \colon \T_\rho^d \to \C^{n \times n}$, and a map $\Lambda \colon \T_\rho^d \to \C^{n \times n}$ such that
\begin{align*}
    E_\text{red}(\theta - \omega) &= P_2(\theta) M_0(\theta - \omega) P_1(\theta - \omega) - \Lambda (\theta - \omega),\\
     E_\text{inv}(\theta) &= \Id - P_2(\theta) P_1(\theta),
\end{align*}
where $\Lambda=\text{diag}(\Lambda_s, \Lambda_u)$ (see Section \ref{subsec:transfer}), such that $\NORM{\Lambda_s}_\rho \leq \lambda_s <1$ and $\NORM{\Lambda_u^{-1}}_\rho \leq \lambda_u <1$. Define
\[
\lambda = \max \left( \lambda_s , 2- \frac{1}{\lambda_u} \right).
\]
Assume also that the error in reducibility, along with the error of invertibility of the map $P_1$, is rather small;
\begin{enumerate}[a)]
    \item $\|E_\text{red}(\theta)\|_\rho = \ep_1 < 1$;\label{redErr}
    \item  $\|E_\text{inv}(\theta)\|_\rho = \ep_2 < 1$,\label{invErr}
\end{enumerate}
and moreover
\begin{enumerate}[c)]
    \item $\lambda+\ep_1 + \ep_2 < 1$. \label{lemma:condC}
\end{enumerate}
Then, $\NORM{(\calM - \calI)^{-1}}_\rho \leq \sigma$ where
\begin{equation*}
    \sigma = \NORM{P_1}_\rho \frac{1}{1 - (\lambda + \ep_1 + \ep_2)} \NORM{P_2}_\rho .
\end{equation*}
\end{lemma}

\begin{proof}

From \ref{invErr}, we can induce that both maps $P_1$ and $P_2$ are invertible. Since $E_\text{inv} = \Id -P_2 P_1$ and therefore $P_2P_1 = \Id - E_\text{inv}$, using a Neumann series argument given \ref{invErr}, both $\Id - E_\text{inv}$ and $P_2P_1$ are invertible and
\begin{equation*}
    P_1^{-1}=(\Id - E_\text{inv})^{-1} P_2 \  , \quad P_2^{-1}=P_1(\Id - E_\text{inv})^{-1} .
\end{equation*}

We need now to express the reducibility error equation in terms of their functionals, defined as follows

\begin{itemize}
    \item $\calP_1 \Delta (\theta) = P_1(\theta) \Delta(\theta)$;
    \item $\calP_2 \Delta (\theta) = P_2(\theta) \Delta(\theta)$;
    \item $\calM \Delta (\theta) = M(\theta - \omega) \Delta(\theta - \omega)$;
    \item $\calL_s \Delta_s (\theta) = \Lambda_s(\theta - \omega) \Delta_s(\theta - \omega)$;
    \item $\calL_u \Delta_u (\theta) = \Lambda_u(\theta - \omega) \Delta_u(\theta - \omega)$;
    \item $\calE_\text{red} \Delta (\theta) = E_\text{red} (\theta - \omega) \Delta (\theta - \omega)$;
    \item $\calE_\text{inv} \Delta (\theta) = E_\text{inv} (\theta) \Delta (\theta)$.
\end{itemize}
With this, we can move from the cocycle notation to the transfer operator notation.
\begin{align*}
    \calE_\text{red} &= \calP_2 \calM \calP_1 - \calL \nonumber \\
     &= \calP_2 (\calM - \calI) \calP_1 + \calP_2 \calP_1 - \calL + \calI - \calI \nonumber \\
    &= \calP_2 (\calM - \calI) \calP_1 - \calE_\text{inv} - (\calL - \calI) ,
    \end{align*}
which implies
\begin{equation*}
        \calE_\text{inv} + \calE_\text{red} = \calP_2 (\calM - \calI) \calP_1 - (\calL - \calI),
\end{equation*}
from where
    \begin{align}\label{eq::hyperInvertible}
    \calM - \calI &= \calP_2^{-1} \big( (\calL - \calI) + \calE_\text{inv} + \calE_\text{red} \big) \calP_1^{-1} \nonumber \\
    &= -\calP_2^{-1} \big[ \calI - (\calL + \calE_\text{red} + \calE_\text{inv}) \big] \calP_1^{-1}.
\end{align}
Now, if $\NORM{\Lambda_s}_\rho \leq \lambda_s <1$ and $\NORM{\Lambda_u^{-1}}_\rho \leq \lambda_u <1$, by the fixed-point theorem, $\NORM{(\calL_s-\calI)^{-1}}_\rho \leq \frac{1}{1-\lambda_s}$ and $\NORM{(\calL_u-\calI)^{-1}}_\rho \leq \frac{\lambda_u}{1-\lambda_u}$, and therefore,
\begin{equation*}
    \NORM{(\calL - \calI)^{-1}}_\rho \leq \max\left( \frac{1}{1-\lambda_s}, \frac{\lambda_u}{1-\lambda_u} \right) = \frac{1}{1 - \lambda}
\end{equation*}
given our definition of $\lambda$. By \ref{lemma:condC}, the whole equation \eqref{eq::hyperInvertible} is invertible and
\begin{align*}
    (\calM - \calI)^{-1} &= -\calP_1 \big[ \calI - (\calL + \calE_\text{red} + \calE_\text{inv}) \big]^{-1} \calP_2 \ .
\end{align*}
Giving as a result
\begin{align*}
    \NORM{(\calM - \calI)^{-1}}_\rho \leq \NORM{\calP_1}_\rho \frac{1}{1 - (\lambda + \ep_1 + \ep_2)} \NORM{\calP_2}_\rho \leq \NORM{P_1}_\rho \frac{1}{1 - (\lambda + \ep_1 + \ep_2)} \NORM{P_2}_\rho.
\end{align*}
\end{proof}

\section{Fourier Series Estimates}\label{sec::Fourier}
The heart of our computer-assisted methodology, as presented in this work, revolves around bounding the error that arises when approximating a periodic function using its discrete Fourier transform. This problem, which naturally arises in the field of approximation theory, has long been recognized for its significance. Motivated by the context of our present paper, we specifically tackle this problem for analytic functions. The estimates presented in this section have been borrowed from those provided in \cite{FiguerasHaroLuque2017} and adapted for this specific scenario (such as our simplification in even-sized grids or the 1D case implementation). Nonetheless, for the sake of completeness, we present here the most general and relevant results. For more details and complete proofs, refer to~\cite{FiguerasHaroLuque2017}.

\subsection{Discretization of the Torus and Fourier
Transforms}\label{ssec:notation:dft}

Given a real-analytic function $u: \T^d_{\hatrho} \rightarrow \C$,
we consider its Fourier series
\[
u(\theta)=\sum_{k \in \Z^d} \hat{u}_k \ee^{2 \pi \ii k \cdot \theta},
\]
where the Fourier coefficients are given by the \emph{Fourier transform} (${\rm
FT}$)
\begin{equation}\label{eq:fourier:coef}
\hat{u}_k = \int_{[0,1]^d} u(\theta) \ee^{-2 \pi \ii k \cdot \theta} \dif
\theta.
\end{equation}
\begin{remark}
    Notice that the Fourier coefficients satisfy the symmetry $u^*_k = u_{-k}$, where $u^*_k$ denotes the complex conjugate of $u_{k}$. Such coefficients decrease exponentially fast, which will prove a useful property later on.
\end{remark}
For $\rho < \hatrho$, its \emph{Fourier norm} is given by
\begin{equation*}
\NORM{u}_{F,\rho}=\sum_{k \in \Z^d} | \hat{u}_k| \ee^{2\pi |k|_1 \rho} < \infty,
\end{equation*}
where $|k|_1 = \sum_{i= 1}^d |k_i|$.  We observe that $\NORM{u}_{\rho} \leq
\NORM{u}_{F,\rho} < \infty$. In practical applications, the Fourier norm can be very effective in order to bound the supremum norm on a complex strip, as we will see in Section \ref{sec::methodology}.

We consider a sample of points on the regular grid of size $\NF=(\Ni{1},\ldots,\Ni{d}) \in \N^d$
\begin{equation*}
\theta_j:=(\theta_{j_1},\ldots,\theta_{j_d})=
\left(\frac{j_1}{\Ni{1}},\ldots,
\frac{j_d}{\Ni{d}}\right),
\end{equation*}
where $j= (j_1,\ldots,j_d)$, with $0\leq j_\ell < \Ni{\ell}$ and $1\leq \ell
\leq d$. This defines a $d$-dimensional sampling $\{u_j\}$, with
$u_j=u(\theta_j)$. The total number of points is  $\Ntot = \Ni{1} \cdots
\Ni{d}$.  The integrals in Equation~\eqref{eq:fourier:coef} are approximated
using the trapezoidal rule on the regular grid, obtaining the discrete Fourier
transform (DFT)
\[
\tilde{u}_k= \frac{1}{\Ntot} \sum_{0\leq j < \NF} u_j \ee^{-2\pi
\ii k \cdot \theta_j},
\]
where the sum runs over integer subindices $j \in \Z^d$ such that $0\leq
j_\ell < \Ni{\ell}$ for $\ell= 1,\dots, d$.  Notice that $\tilde u_k$ is
periodic with respect to the components $k_1,\dots, k_d$ of $k$, with
periods $\Ni{1}, \dots, \Ni{d}$, respectively.
The periodic function $u$ is approximated by the discrete Fourier approximation (which is an approximation of the truncated Fourier series),
\begin{equation*}
\tilde u(\theta)= \sum_{k \in \INF} \tilde{u}_k \ee^{2 \pi \ii k \cdot
\theta},
\end{equation*}
where $\INF$ is the finite set of multi-indices given by
\begin{equation*}
 \INF= \bigg\{ k \in \Z^d \st -\frac{\Ni{\ell}}{2} \leq k_\ell <
 \frac{\Ni{\ell}}{2}, 1\leq \ell \leq d \bigg\}.
\end{equation*}
Along this section, we will use the standard notation $[x]$ for the integer part
of $x$: $[x]=\min\left\{j\in\mathbb Z \st x\leq j\right\}$.

\begin{remark}
As we have previously stated, the Fourier coefficients are symmetrical, that is, $\hat{u}_k^* = \hat{u}_{-k}$, which holds for the DFT coefficients as well. This presents a problem regarding the way the DFT is defined. See that since we are treating the real-analytic case, our function $u$ evaluated over the points of the grid will acquire real values, but depending on the parity of the size of the grid, $N_{F,\ell}$, for $0 \leq \ell \leq d$, the discrete approximation outside the grid will not. The reason behind this phenomenon lies on the fact that if $N_{F,\ell}$ is odd, due to the coefficients' symmetry, the resulting function will remain real, but if $N_{F,\ell}$ is even, then $N_{F,\ell}-1$ is odd, which means that the term $- \left[ \frac{N_{F,\ell}}{2} \right]$ of the sum, called the Nyquist term, will be unpaired. The lack of its symmetrical pair results on a complex function whose derivative will have the imaginary term $\ii$. This does not present a major issue since the Nyquist term will naturally be very small. Nonetheless, if it is desired to look for a way to express the function $u$ in terms of its DFT without this little problem, one shall eliminate the Nyquist term. Since doing this does not come without a cost of breaking the grid-Fourier correspondence, we will pre-process our data such that the approximate $K$ is taken such that its Nyquist term is already zero.
\end{remark}

\subsection{Error Estimates on the Approximation of Analytic Periodic Functions}

Based on the focus of our paper, we work with real-analytic functions in spaces defined on a complex strip of the torus (see Section \ref{sec:setting}). However, the arguments presented can be adapted to other spaces as well. Our main goal is to accurately quantify the error between $\tilde{u}$ and $u$ by using appropriate norms. Next, we state the main result of this section, that allows us to control the
error between $\tilde u$ and $u$.

\begin{theorem}
\label{FourierEstimate} Let $u \in A(\bar{\T}^d_{\hatrho},\C)$, with $\hatrho>0$.
Let $\tilde u$ be the discrete Fourier approximation of $u$ in the regular
grid of size
$\NF= (\Ni{1},\dots, \Ni{d}) \in \N^d$.  Then,
\[
\NORM{\tilde u-u}_\rho \leq C_{\NF}(\rho, \hatrho) \NORM{u}_{\hatrho} \ ,
\]
for $0\leq \rho < \hatrho$,
where $C_{\NF}(\rho, \hatrho)= S_\NF^{*1}(\rho,\hatrho) + S_\NF^{*2}(\rho,\hatrho)  +
T_\NF(\rho,\hatrho)$ is given by
\[
S_\NF^{*1}(\rho,\hatrho) =
\prod_{\ell= 1}^d \frac{1}{1-\ee^{-2\pi  \hatrho\Ni{\ell} }}
\sum_{\begin{array}{c} \sigma\in \{-1,1\}^d \\ \sigma\neq (1,\dots,1) \end{array}}
\prod_{\ell= 1}^d \ee^{(\sigma_\ell-1)\pi\hatrho \Ni{\ell}} \nuNi{\ell}(\sigma_\ell\hatrho-\rho),
\]
\[
S_\NF^{*2}(\rho,\hatrho) = \prod_{\ell= 1}^d \frac{1}{1-\ee^{-2\pi
\hatrho\Ni{\ell} }}  \left(1- \prod_{\ell= 1}^d  \left(1-\ee^{-2\pi
\hatrho\Ni{\ell} }\right)\right) \prod_{\ell= 1}^d \nuNi{\ell}(\hatrho-\rho)
\]
and
\[
T_\NF(\rho,\hatrho)= \left(  \frac{\ee^{2\pi (\hatrho-\rho)} + 1}{\ee^{2\pi
(\hatrho-\rho)} -1} \right)^d \ \left( 1 - \prod_{\ell= 1}^d \left(1-
\muNi{\ell}(\hatrho-\rho)\ e^{-\pi(\hatrho-\rho) \Ni{\ell}} \right) \right),
\]
with
\[
\nuNi{\ell}(\delta)= \frac{\ee^{2\pi \delta} + 1 }{\ee^{2\pi \delta} -1}
\left(1- \muNi{\ell}(\delta) \ \ee^{-\pi \delta \Ni{\ell}}\right) \quad
\mbox{and} \quad \muNi{\ell}(\delta) =
\begin{cases}
\ 1 &\mbox{if $\Ni{\ell}$ is even}, \\ \displaystyle \frac{2
\ee^{\pi\delta}}{\ee^{2\pi\delta}+1} &\mbox{if $\Ni{\ell}$ is odd}.
\end{cases}
\]
\end{theorem}

\begin{remark}
As remarked in \cite{FiguerasHaroLuque2017},
\[
C_{\NF}(\rho,\hatrho)
\simeq O(\ee^{-\pi (\hat \rho-\rho) \min_{\ell} \{\Ni{\ell}\}}) \ ,
\]
which implies that the error caused by DFT approximation should be very small if $\hatrho$ is chosen appropriately. This means that working with a DFT-approximated version of our function should provide very accurate results when properly implemented with interval arithmetics.
\end{remark}

\begin{remark}
    Although \thmref{FourierEstimate} is presented in the case of a $d$-dimensional torus, the estimates can be greatly simplified when working with a $1$-dimensional torus, as  will be our case. Notice also that if we choose $\NF$ to be even (see \cite{FiguerasHaroLuque2017} for other scenarios), the estimates simplify further to give the terms

    \begin{align*}
S_{\NF}^{*1}(\rho, \hatrho)&= \frac{\ee^{-2 \pi \hatrho \NF}}{1-\ee^{-2 \pi \hatrho \NF}}
\frac{\ee^{-2\pi (\hatrho + \rho)}+1}{\ee^{-2 \pi (\hatrho+\rho)}-1} \left( 1 - \ee^{\pi (\hatrho + \rho) \NF} \right) ,\\
S_{\NF}^{*2}(\rho, \hatrho)&= \frac{\ee^{-2 \pi \hatrho \NF}}{1-\ee^{-2 \pi \hatrho \NF}}
\frac{\ee^{2\pi (\hatrho - \rho)}+1}{\ee^{2 \pi (\hatrho-\rho)}-1} \left( 1 - \ee^{-\pi (\hatrho - \rho) \NF} \right) ,\\
T_{\NF}(\rho, \hatrho)&= \frac{\ee^{2\pi (\hatrho - \rho)}+1}{\ee^{2 \pi (\hatrho-\rho)}-1} \ee^{-\pi (\hatrho - \rho) \NF} \ .
\end{align*}
\end{remark}

\subsection{Results on Matrices of Periodic Functions}\label{ssec:dft:matrices}

In this section, we consider some extensions of \thmref{FourierEstimate} to deal
with matrix functions $A \in A(\bar\T^d_{\hatrho},\C^{m_1 \times m_2})$. Our goal
is to control the propagation of the error when we perform matrix operations.
Specifically, we are interested in the study of products and inverses,
but the ideas given below can be adapted to control other operations if necessary. The first result is obtained directly from \thmref{FourierEstimate}:

\begin{corollary}\label{cor:matrix:multi}
Let us consider two matrix functions $A \in A(\bar{\T}^d_{\hatrho},\C^{m_1 \times m_2})$, and $B \in A(\bar{\T}^d_{\hatrho},\C^{m_2 \times m_3})$. We denote by $A B$
the product matrix and $\smash{\widetilde{A B}}$ the corresponding
approximation given by DFT. Given a grid of size $\NF=(\Ni{1},\ldots,\Ni{n})$,
we evaluate $A$ and $B$ in the grid, and we interpolate the points $A
B(\theta_j)=A(\theta_j) B(\theta_j)$ to get $\smash{\widetilde{A B}}$.  Then, we have
\begin{equation*}
\NORM{AB-\widetilde{A B}}_\rho \leq C_{\NF}(\rho,\hat \rho) \NORM{A}_{\hat \rho}
\NORM{B}_{\hat \rho} \ ,
\end{equation*}
for every $0 \leq \rho < \hat \rho$,
where $C_{\NF}(\rho,\hat \rho)$ is given in \thmref{FourierEstimate}.
\end{corollary}

Notice that Corollary~\ref{cor:matrix:multi} is useful to control
the product of approximated objects. If $\tilde A$ and
$\tilde B$ are the corresponding approximations of $A$ and $B$ given by DFT,
then
\begin{equation}\label{eq:dtf:matrix:trunc}
\NORM{\tilde A \tilde B-\widetilde{\tilde A \tilde B}}_\rho
\leq C_{\NF}(\rho,\hat \rho) \NORM{\tilde A}_{\hat \rho} \NORM{\tilde B}_{\hat \rho}
\leq C_{\NF}(\rho,\hat \rho) \NORM{\tilde A}_{F,\hat \rho} \NORM{\tilde B}_{F,\hat \rho}
\end{equation}
for every $0 \leq \rho < \hat \rho$.
Notice that since $\tilde A$ and $\tilde B$
are Fourier series with finite support, then
it is interesting to control Equation~\eqref{eq:dtf:matrix:trunc}
using Fourier norms. The second result allows us to control the inverse of a matrix using the
discrete Fourier approximation:

\begin{corollary}\label{cor:matrix:inv}
Let us consider a matrix function $A \in A(\bar{\T}^d_{\hatrho},\C^{m \times m})$.  Given a grid of size
$\NF=(\Ni{1},\ldots,\Ni{n})$, we evaluate $A$ in the grid and compute
the inverses $X(\theta_j)=A(\theta_j)^{-1}$.  Then, if $\tilde X$ is the
corresponding discrete Fourier approximation associated to the sample
$X(\theta_j)$, the error $\Gamma(\theta)=I_m-A(\theta) \tilde
X(\theta)$ satisfies
\begin{equation*}
\NORM{\Gamma}_\rho \leq C_{\NF} (\rho,\hat \rho) \NORM{A}_{\hat \rho} \NORM{\tilde
X}_{\hat \rho} \ ,
\end{equation*}
for $0 \leq \rho < \hat \rho$. Moreover, if $\NORM{\Gamma}_\rho<1$, there exists an inverse $A^{-1} \in A(\bar{\T}^d_{\rho},\C^{m \times m})$
satisfying
\begin{equation*}
\NORM{A^{-1}-\tilde X}_\rho \leq \frac{\NORM{\tilde X}_{\hat \rho}
\NORM{\Gamma}_\rho}{1-\NORM{\Gamma}_\rho} \ .
\end{equation*}
Furthermore,
\begin{align}\label{norm_inverse}
    \NORM{A^{-1}}_\rho &\leq \NORM{A^{-1} - \tilde{X}}_\rho + \NORM{\tilde{X}}_\rho \leq \frac{\NORM{\tilde X}_{\hat \rho}
\NORM{\Gamma}_\rho}{1-\NORM{\Gamma}_\rho} + \NORM{\tilde{X}}_\rho = \frac{\NORM{\tilde{X}}_\rho}{1-\NORM{\Gamma}_\rho} \nonumber \\
&\leq \frac{\NORM{\tilde{X}}_\rho}{1-C_{\NF} (\rho,\hat \rho) \NORM{A}_{\hat \rho} \NORM{\tilde
X}_{\hat \rho}} .
\end{align}
\end{corollary}

\section{Methodology}
\label{sec::methodology}
We introduce now the methodology with which we can calculate the corresponding error bounds when implementing the validation theorem, especially by using the Fourier approximation results previously presented.

A crucial initial step involves meticulously selecting the inputs for validation. In our study, we will work with a real-analytic one-dimensional torus, which will be directly inputted into the algorithm as a truncated Fourier series (although the following procedure can be trivially generalized for a $d$-dimensional torus, adjusting grid sizes and everything else accordingly). These inputs can either originate from data already structured in this form or be prepared through preprocessing for convenience. Notably, the torus and other inputs are derived using the reducibility algorithm outlined in \cite{mamotreto, DelallaveHaro2006}.

For the validation process, we employed the Radix-2 DIT Cooley-Tukey algorithm to compute the Fast Fourier Transform (FFT), necessitating grid sizes that are powers of 2. This choice in algorithm ensures efficient computation and scalability, aligning with the requirements of our work \cite{cooley1965algorithm}. Thus, the Nyquist term is pre-set to 0. Note, though, that the presented formulas work for $N$ both even and odd, with the main difference that for odd $N$ there is no Nyquist term that needs to be taken care of.

In this manner, $K_0(\theta) = \widetilde{K_0(\theta)}$, and has the form
\begin{equation*}
   K_0(\theta)  = \sum\limits_{k=-\left[\frac{N-1}{2}\right]}^{\left[\frac{N-1}{2}\right]} \widetilde{K}_{0,k} \: \ee^{2 \pi \ii k \theta} \: .
\end{equation*}
Which leads easily to $K_0(\theta + \omega)$,
\begin{equation*}
    K_0(\theta + \omega) = \sum\limits_{k=-\left[\frac{N-1}{2}\right]}^{\left[\frac{N-1}{2}\right]} (\widetilde{K}_{0,k} \: \ee^{2 \pi \ii k \omega})  \: \ee^{2 \pi \ii k \theta} \: .
\end{equation*}
In the same manner, $P_1=\widetilde{P}_1$, $P_2 = \widetilde{P}_2$, $\Lambda_s = \Tilde{\Lambda}_s$ and $\Lambda_u = \Tilde{\Lambda}_u$ will be our real-analytic inputs in Fourier space with Nyquist term 0. Notice that since our inputs are already truncated Fourier series, there is no error committed when going to grid space and back via Fourier transform (as the FFT). We will also keep the same notation as in \secref{sec:setting} for our real-analytic map $F$ but changing the support point of the torus to $\theta$, as this will allow to more easily work with the application of rotations. Nonetheless, with the aim of applying the results from \secref{sec::Fourier}, $F$ will be defined over a slightly thicker domain, namely $D_{\hatrho,R}$ for a $\hatrho > \rho$.

\subsection{Boxes Method}\label{Boxes Method}
When it comes to rigorous numerics, interval arithmetic is usually the standard way to proceed. This will allow us to enclose small values that go beyond the computer's representation capabilities in an interval, easing its manipulation. In this manner, and using the aforementioned inputs, we can compute the bounds that appear from $K_0$ in \thmref{thm::validation}. Such bounds require us to compute $\rho$-norms, or supremum norms on a complex neighborhood of the base torus (or more practically, the grid). This can be troublesome, as we would need to evaluate non-truncated objects like, for example, $F(K_0(\theta), \theta)$ over a complex neighborhood of our grid using the computer. In some cases, we can easily estimate such norm by raw bounding using the Fourier norm. But in other cases in which we may not have the desired object in Fourier space, we have to resort to another method. The boxes method allows us to extend the grid to the complex space by creating complex boxes $C_j = \{\theta_j + \varphi \st |\re \:\varphi | \leq \frac{1}{2N}$, $|\im \: \varphi | \leq \rho\}$ (notice that the choice of $\rho$ will depend on the context). These boxes are easily created in the computer as they are the cartesian product of two intervals (the real part and the imaginary part). For the computation of the norm, and following the previous example, first we will have to calculate the image of such boxes through $K$ before applying $F$. Thus, we first need to calculate $K_0(\theta + \varphi)$. Notice that this is no more than rotating the torus as we have done before, but this time the rotation is complex. Again, this is easier done in Fourier series form given the symmetry of the coefficients thanks to the real-analyticity of $K_0$:
\begin{equation*}
    K_0(\theta + \varphi) = \sum\limits_{k=-\left[\frac{N-1}{2}\right]}^{\left[\frac{N-1}{2}\right]} (\widetilde{K}_{0,k} \: \ee^{2 \pi i k \varphi})  \: \ee^{2 \pi i k \theta} = \sum\limits_{k=-\left[\frac{N-1}{2}\right]}^{\left[\frac{N-1}{2}\right]} \left(\frac{\widetilde{K}_{0,k}}{\ee^{2 \pi k \im \varphi}} \: \ee^{2 \pi i k \re \varphi}\right)  \: \ee^{2 \pi i k \theta} \: .
\end{equation*}
Finally, we just have to go back to grid space via IFFT. With this new torus, we can calculate $F(K(C_j), C_j)$ for each $j$. The norm of the object will then consist of the supremum out of all the new boxes.

\subsection{The Invariance Error Bound}

For the first bound, we can write the error produced in the invariance equation as
\begin{equation*}
E(K_0)(\theta) = F(K_0(\theta), \theta) - K_0(\theta + \omega) \: .
\end{equation*}

Since the Fourier series of $F(K_0(\theta), \theta)$ is an infinite sum, we would like to approximate it by a finite sum where, in an abuse of notation, $\widetilde{F}_{0,k}$ will indicate the $k$th Fourier coefficient of $F(K_0(\theta), \theta)$;
\begin{equation*}
    \reallywidetilde{F(K_0(\theta), \theta)} = \sum\limits_{k=-\left[\frac{N}{2}\right]}^{\left[\frac{N-1}{2}\right]} \widetilde{F}_{0,k} \: \ee^{2 \pi \ii k \theta} \: .
\end{equation*}
This allows us to split the computation of the error into two parts:
\begin{equation*}
    \|E(K_0)(\theta)\|_\rho \leq \| F(K_0(\theta), \theta ) - \reallywidetilde{F(K_0(\theta), \theta)} \|_\rho+ \| \reallywidetilde{F(K_0(\theta), \theta)} - K_0(\theta + \omega) \|_\rho \leq \ep \: ,
\end{equation*}
where the last term of the sum can be easily bounded by the Fourier norm:
$$ \| \reallywidetilde{F(K_0(\theta), \theta)} - K_0(\theta + \omega) \|_\rho \leq  \| \reallywidetilde{F(K_0(\theta), \theta)} - K_0(\theta + \omega) \|_{F, \rho},$$
since it is a difference of truncated Fourier series and hence a truncated Fourier series with differences in the coefficients.
For the first term and for a given $\hatrho>\rho$ such that $\calK_{\hatrho} \subset D_{\hatrho,R}$, where recall $D_{\hatrho,R}$ is the domain of $F$, using \thmref{FourierEstimate}, we have
\begin{equation*}
    \| F(K_0(\theta), \theta) - \reallywidetilde{F(K_0(\theta), \theta)} \|_\rho \leq C_N(\rho, \hatrho) \: \|F(K_0(\theta), \theta )\|_{\hatrho} \ ,
\end{equation*}
where the last term will require the boxes method described before in Subsection \ref{Boxes Method} and a suitable choice for $\hatrho$.

\subsection{The Reducibility and Invertibility Error Bounds}
The next bound to be computed is the reducibility error bound, where the reducibility error is given by
\begin{equation*}
    E_\text{red}(\theta) = P_2(\theta + \omega) M_0(\theta) P_1(\theta) - \Lambda (\theta) \: ,
\end{equation*}
with $M_0(\theta) = \Dif_z F(K_0(\theta), \theta)$. Recall that we already have $P_1$, $P_2$, $\Lambda_s$ and $\Lambda_u$ (and therefore $\lambda$) as inputs. Using \corref{cor:matrix:multi}, the norm of the reducibility error is then
\begin{align*}
    \| E_\text{red} (\theta)\|_\rho &= \| P_2(\theta+\omega) M_0(\theta) P_1(\theta) - \Lambda (\theta) \|_\rho \\
    &\leq \|P_2(\theta+\omega) M_0(\theta) P_1(\theta) - \reallywidetilde{P_2(\theta+\omega) M_0(\theta) P_1(\theta)}\|_\rho \\
    &\qquad + \|\reallywidetilde{P_2(\theta+\omega) M_0(\theta) P_1(\theta)} - \Lambda(\theta)\|_\rho \\
    &\leq C_N(\rho, \hatrho) \|P_2(\theta)\|_{\hatrho} \|M_0(\theta)\|_{\hatrho} \|P_1(\theta)\|_{\hatrho}\\
    &\qquad + \|\reallywidetilde{P_2(\theta+\omega) M_0(\theta) P_1(\theta)} - \Lambda(\theta)\|_{F,\rho} \\
    &= \ep_1 ,
\end{align*}
where the last term remains the same since it is the norm of the difference of a matrix of truncated Fourier series, and a constant matrix (given that $\Lambda(\theta)$ is an input).
\begin{remark}
    Every time that a result of \thmref{FourierEstimate} is used (such as the derived corollaries), a certain $\hatrho > \rho$ is chosen. This means that for every computation of a bound  in which we use those results, a different $\hatrho$ can be chosen. However, we will play with $\rho$ and $\hatrho$ in order to find the optimal value of $C_N(\rho, \hatrho)$, and once found, it will be used in every calculation, meaning that we will keep the same choice of $\hatrho$ throughout.
\end{remark}
For the purpose of finding the $\sigma$ bound in the validation theorem, we require the value of $\lambda$. In our case, the $\Lambda$ matrix will be a constant matrix as explained in Section \ref{sec:setting} and therefore the value $\lambda$ is given. However, this is not always the case, as other types of fixed points may lead to non-constant $\Lambda(\theta)$, requiring then a way to estimate a $\lambda$ value. For such purpose (and although it is not necessary for our case), we illustrate here how to find such a value.
Notice that
\begin{equation*}
    \|(\Lambda_u)^{-1}\|_\rho \leq \|(\Lambda_u)^{-1} - \widetilde{(\Lambda_u)^{-1}}\|_\rho + \|\widetilde{(\Lambda_u)^{-1}}\|_\rho \leq \|(\Lambda_u)^{-1} - \widetilde{(\Lambda_u)^{-1}}\|_\rho + \|\widetilde{(\Lambda_u)^{-1}}\|_{F, \rho} \: .
\end{equation*}
By \corref{cor:matrix:inv},
\begin{equation*}
    \|(\Lambda_u)^{-1} - \widetilde{(\Lambda_u)^{-1}}\|_\rho \leq \frac{\|\widetilde{(\Lambda_u)^{-1}}\|_{\hatrho} \: \|\Gamma\|_\rho}{1- \|\Gamma\|_\rho}
\end{equation*}
where $\Gamma(\theta) = \Id_{n_U} - \Lambda_u(\theta) \: \reallywidetilde{(\Lambda_u(\theta))^{-1}}$ as used in \corref{cor:matrix:inv}, with
\begin{equation*}
    \|\Gamma\|_{\rho} \leq C_N(\rho, \hatrho) \: \|\Lambda_u\|_{\hatrho} \: \| \widetilde{(\Lambda_u)^{-1}} \|_{\hatrho} \: .
\end{equation*}
All together
\begin{align*}
    \|(\Lambda_u)^{-1} - \widetilde{(\Lambda_u)^{-1}} \|_\rho
    &\leq  \frac{C_N(\rho, \hatrho) \: \|\Lambda_u\|_{F, \hatrho} \: \|\widetilde{(\Lambda_u)^{-1}}\|_{F, \hatrho}^2}{1 - C_N(\rho, \hatrho) \: \|\Lambda_u\|_{F, \hatrho} \: \|\widetilde{(\Lambda_u)^{-1}}\|_{F, \hatrho}} \: .
\end{align*}
The inequality holds due to the fact that $\Lambda_u$ is an input and already a matrix of truncated Fourier series. Thus, following formula \eqref{norm_inverse}, we can take
\begin{equation*}
    \|(\Lambda_u)^{-1}\|_\rho \leq \frac{\|\widetilde{(\Lambda_u)^{-1}}\|_{F, \hatrho}}{1 - C_N(\rho, \hatrho) \: \|\Lambda_u\|_{F, \hatrho} \: \|\widetilde{(\Lambda_u)^{-1}}\|_{F, \hatrho}} = \lambda_u .
\end{equation*}

Next, we seek a bound for the invertibility error. Its norm can be computed using the same procedure:
    \begin{align*}
    \|E_\text{inv}(\theta) \|_\rho &\leq  \|P_2(\theta) P_1(\theta) - \reallywidetilde{P_2(\theta) P_1(\theta)}\|_\rho + \|\reallywidetilde{P_2(\theta) P_1(\theta)} - \Id\|_\rho \\
    &\leq C_N(\rho, \hatrho) \: \|P_2(\theta)\|_{\hatrho} \: \|P_1(\theta)\|_{\hatrho} +\|\reallywidetilde{P_2(\theta) P_1(\theta)} - \Id\|_{F,\rho} \\
    &= \ep_2 .
\end{align*}

\subsection{Computation of $b(r)$}
Only remains the last condition of \thmref{thm::validation}. Recall that condition \ref{condC} required $\forall (z, \theta) \in \Bar{D}_{\rho,r}$, $\NORM{\Dif_z F(z, \theta) - \Dif_z F(K_0(\theta), \theta)}_\rho \leq b(r) \norm{z-K_0(\theta)}$. This means that we need a way of computing such $b(r)$. Notice that for all $\theta \in \bar{\T}_{\rho}$
\begin{align*}
    \| \Dif_z &F(z, \theta) - \Dif_z F(K_0(\theta), \theta)\|_\rho\\
    &= \left\| \int_0^1 \frac{\dif}{\dif t} \Dif_z F(K_0(\theta) + t(z-K_0(\theta)), \theta) \ \dif t \right\| _\rho \\
    &\leq \int_0^1 | \Dif^2_z F(K_0(\theta) + t(z-K_0(\theta)), \theta) | \dif t \cdot | z - K_0(\theta) | \\
    &\leq \sup_{\substack{\theta \in \bar{\T}_\rho \\ |y-K_0(\theta)|\leq r}} |\Dif^2_z F(y, \theta)| \cdot | z - K_0(\theta) | .
\end{align*}
So, we can take
\begin{equation*}
    b(r) = \max_{i=1,\ldots,n} \sup_{(y, \theta) \in \Bar{D}_{\hatrho,R}} |\Dif^2_z F^i(y, \theta)| ,
\end{equation*}
where $F^i$ represents the $i$th component of $F$, as the maximum of the norms of each component of $F$ as bilinear maps. In fact, we compute $b(R)$, as $b(r) \leq b(R)$. The second derivative of $F$ is then evaluated over points in the domain. Such domain is the points $R$-close to the approximate torus over a $\rho$-thick complex base torus, $\Bar{\T}_{\rho}$. This means that, in our case, we will have to apply the boxes method to compute the second derivative, that is, thickening the torus by $R$ in all directions after thickening it by $\rho$ in the complex direction.

\section{Algorithm and Results}
The computations explained in Section \ref{sec::methodology} have been implemented in C++ using interval arithmetics. Two different classes have been created for such purpose: \textit{ComplexInterval}, which has two members, an interval representing the real part and an interval representing the imaginary part of a complex value, and \textit{RealInterval}, which comprises  a single member, an interval. Methods for each class have been implemented such that basic operations and operator overloads are covered, easing the way into building more complex functions as the FFT.

\subsection{Case of Study}
The objects that will serve as starting points for the validation are approximate tori computed through a reducibility method (see \cite{DelallaveHaro2006} for more details). In this study, we take $f$ to be an irrational rotation as mentioned and $F$ as the perturbed standard map. Thus, $\hat{F}$ is $(F, f) : \mathbb{C}^2 \times \mathbb{T}_\C \rightarrow \mathbb{C}^2 \times \mathbb{T}_\C $ given by
\begin{align}
\begin{cases}
f(\theta) = \theta + \omega, \\
F^1(x, y) = x+y - \frac{\kappa}{2 \pi} \sin(2 \pi x) - \epsilon \sin(2 \pi \theta),\\
F^2(x, y) = y - \frac{\kappa}{2 \pi} \sin(2 \pi x) - \epsilon \sin(2 \pi \theta) .
\end{cases}
\label{eq:stdMap}
\end{align}
For a given rotation $\omega\in\R\setminus\Q$, and fixed $\kappa,\epsilon\in \R$, we search for a real-analytic parameterization of a torus $K: \bar\T_\rho \to \C^2$, joint with changes of variable matrices $P_2, P_1: \bar\T_\rho \to \C^{2\times2}$ that reduce the linearized dynamics to $\Lambda: \bar\T_\rho \to \C^{2 \times 2}$ such that $\Lambda = \text{diag} (\Lambda_s, \Lambda_u)$ is constant, and $\lambda_s = |\Lambda_s|$ and $\lambda_u = |\Lambda_u^{-1}|$. That is, we plan to solve the equations
\begin{itemize}
    \item $F(K(\theta - \omega), \theta - \omega) - K(\theta) = 0$ ,
    \item $\Dif F(K(\theta - \omega), \theta - \omega) P_1(\theta - \omega) - P_1(\theta) \Lambda = 0$ ,
    \item $\Id - P_1(\theta) P_2(\theta) = 0$ .
\end{itemize}
These are the equations that of course lead to $E$, $E_\text{red}$, and $E_\text{inv}$ when applied to an approximation $K_0$.

As for the rotation value, we take the golden mean $\omega = \frac{1}{2}(\sqrt{5}-1)$, which is the ``most irrational number" that is used in many works on the existence of quasi-periodic motion (see, e.g., the pioneering work \cite{greene1979method} and other references mentioned in the introduction). Fixing $\kappa \in (0,4)$, for $\epsilon =0$, $K_0(\theta) = (\frac{1}{2}, 0)$ is a fiberwise hyperbolic invariant torus. For $\epsilon>0$ small enough, the torus maintains its hyperbolicity up to a certain critical value $\epsilon_c$, where the linear dynamics of the invariant torus are no longer reducible to diagonal (this is known as the phenomenon of bundle collapse, implying the destruction of the invariant curve, see \cite{haroBreakdown, DelallaveHaro2006, haroTrilogy}). Hence, the validation is possible, in principle, for $\epsilon \in [0, \epsilon_c)$.

In our case, we fixed $\kappa = 1.3$ such that the invariant tori start to lose hyperbolicity as $\epsilon$ approaches a critical value of $\epsilon_c \approx 1.2352755$, as explained in \cite{FiguerasHaro2012,mamotreto}, which means that the validation will get more nuanced and probably require a finer tuning of the parameters such as grid size, width of the complex band, and so on. This can be easily seen on Figure \ref{fig:tori}, as the graph of the torus starts to fractalize as it approaches the critical value. These are the tori that are validated in the following section.

\begin{figure}[htbp]
    \centering
    \begin{subfigure}{0.45\textwidth}
        \centering
        \includegraphics[width=.89\textwidth]{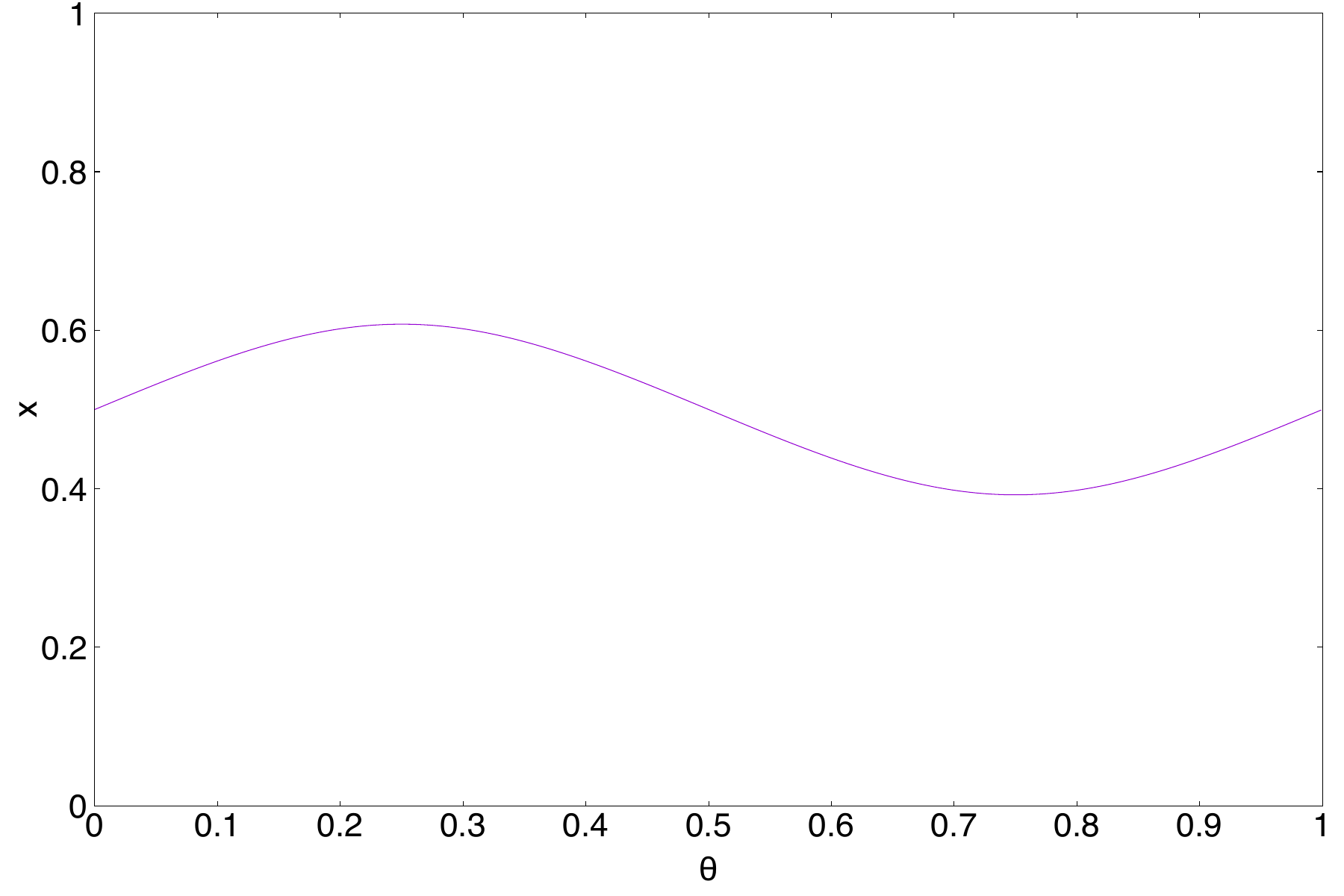}
        \caption{x-projection of the torus, $\epsilon = 0.5$.}
        \label{fig:1a}
    \end{subfigure}
    \hspace{0.5cm}
    \begin{subfigure}{0.45\textwidth}
        \centering
        \includegraphics[width=.89\textwidth]{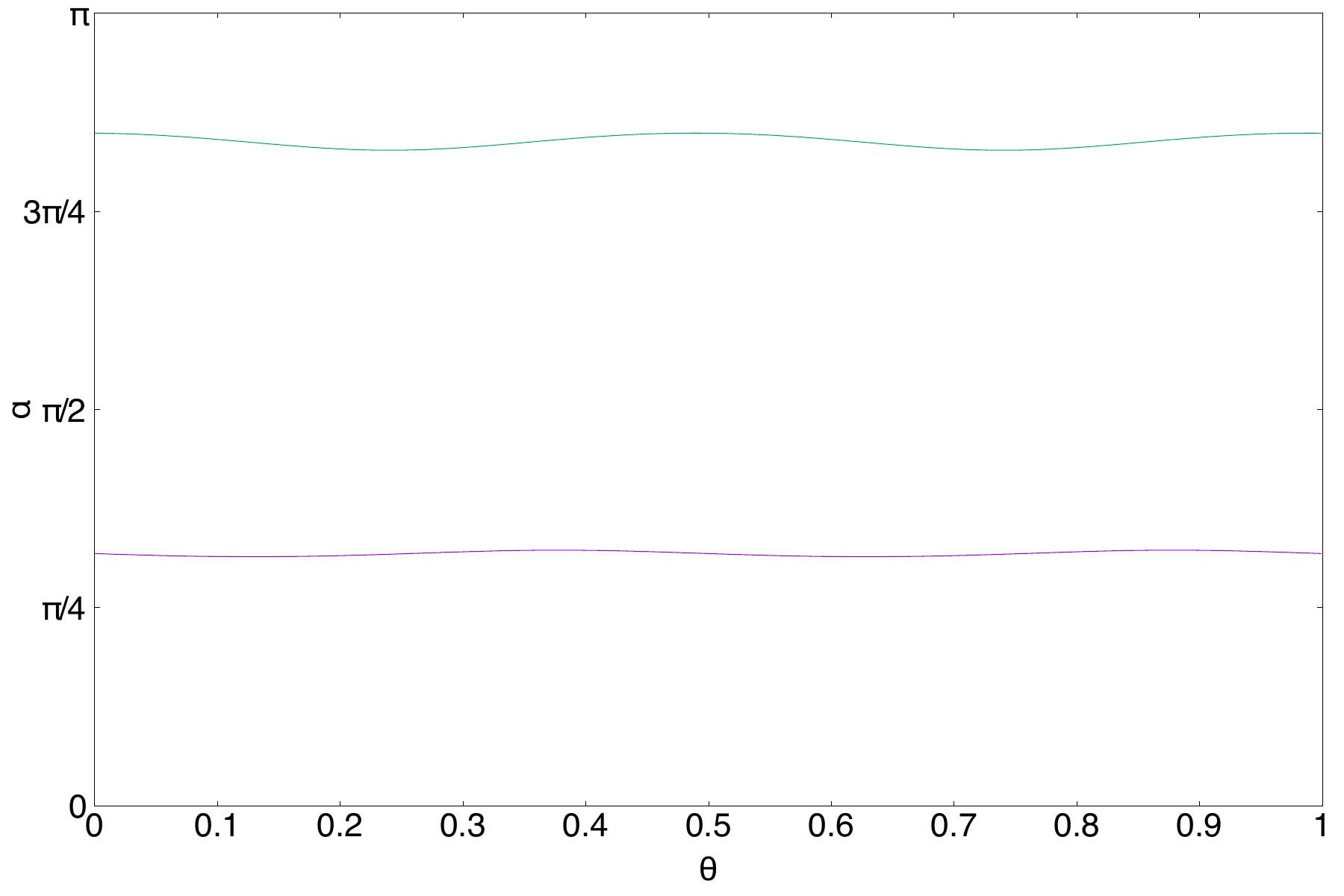}
        \caption{Invariant subbundles, $\epsilon = 0.5$.}
        \label{fig:1b}
    \end{subfigure}

    \medskip

    \begin{subfigure}{0.45\textwidth}
        \centering
        \includegraphics[width=.89\textwidth]{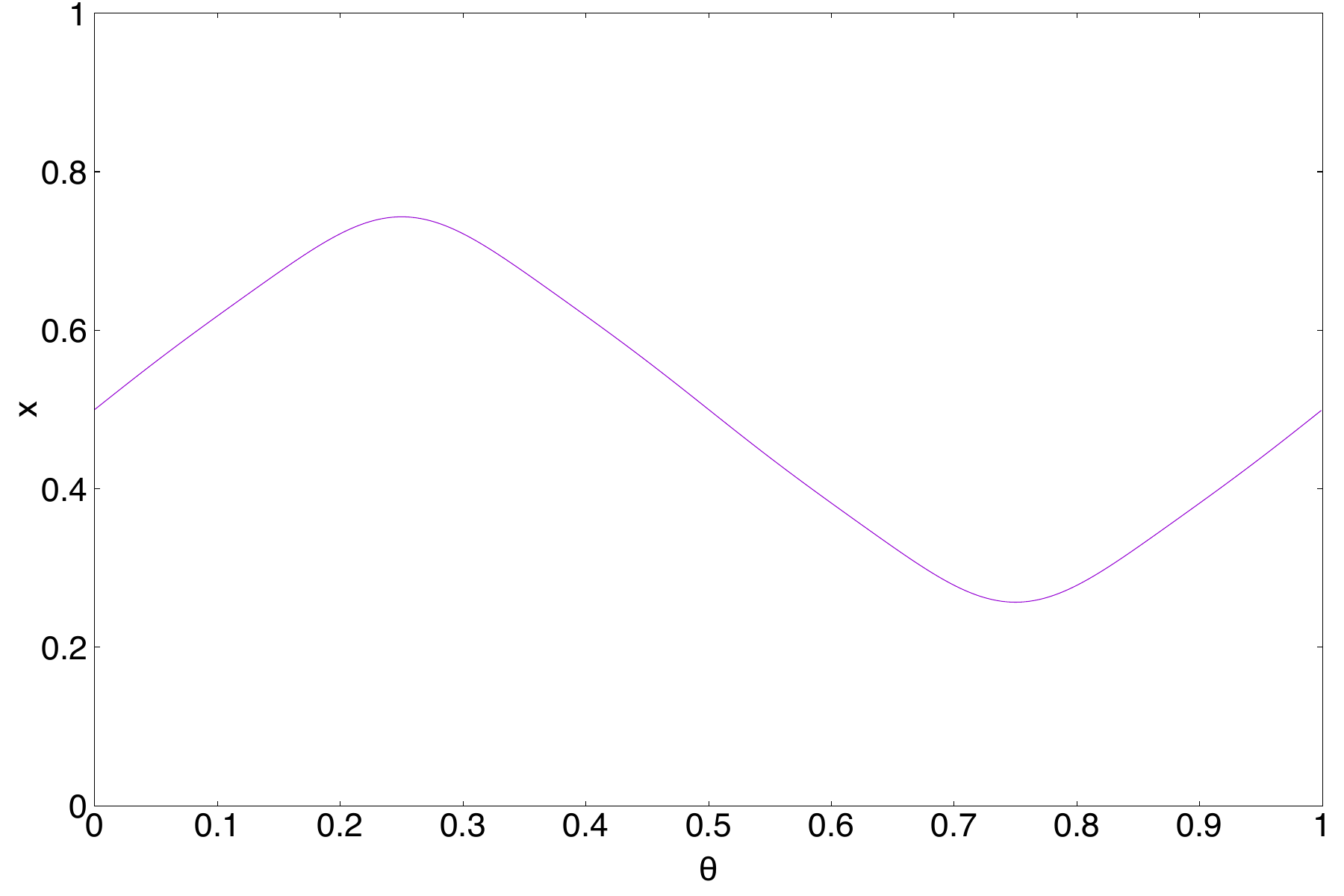}
        \caption{x-projection of the torus, $\epsilon = 1$.}
        \label{fig:1c}
    \end{subfigure}
    \hspace{0.5cm}
    \begin{subfigure}{0.45\textwidth}
        \centering
        \includegraphics[width=.89\textwidth]{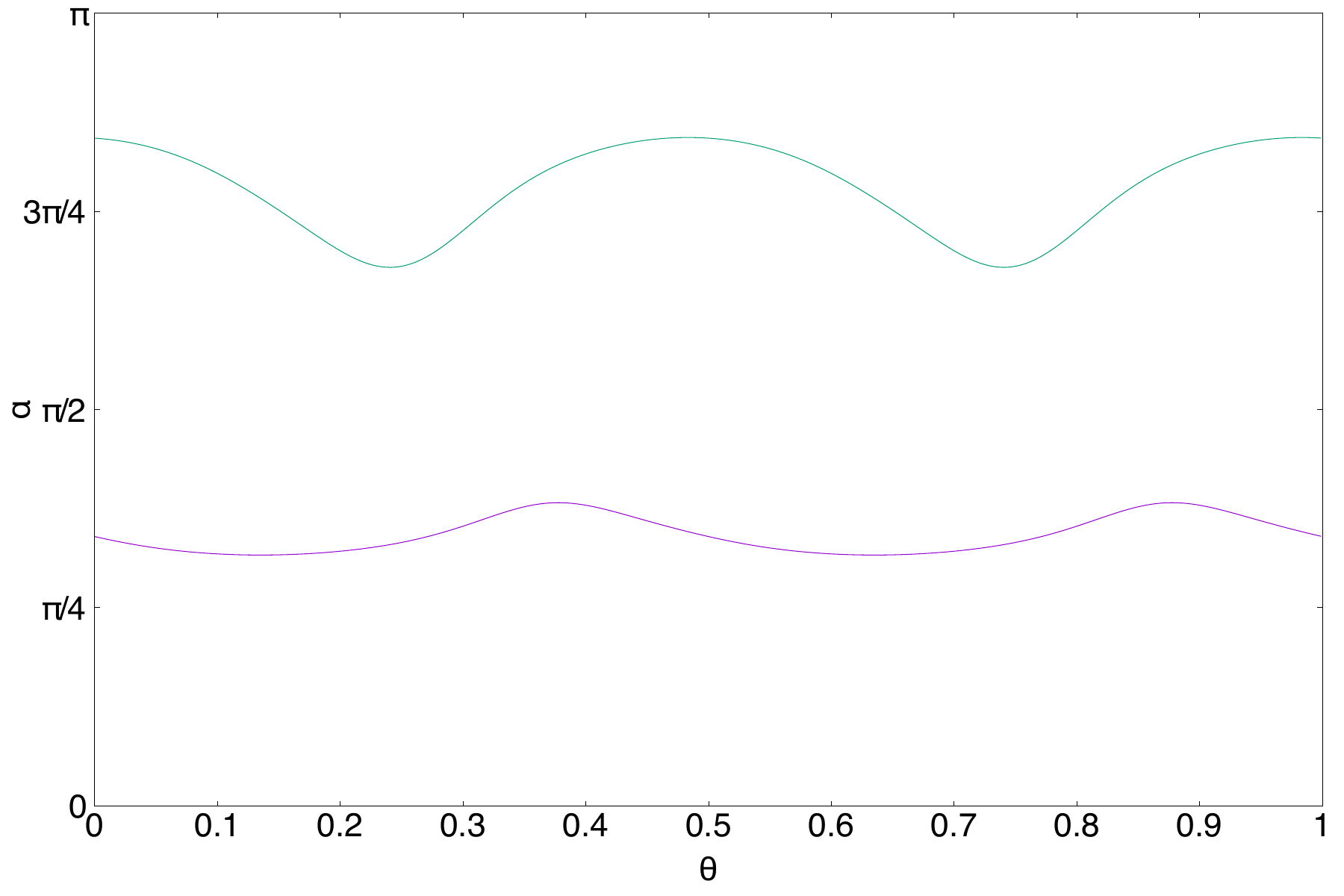}
        \caption{Invariant subbundles, $\epsilon = 1$.}
        \label{fig:1d}
    \end{subfigure}

    \medskip

    \begin{subfigure}{0.45\textwidth}
        \centering
        \includegraphics[width=.89\textwidth]{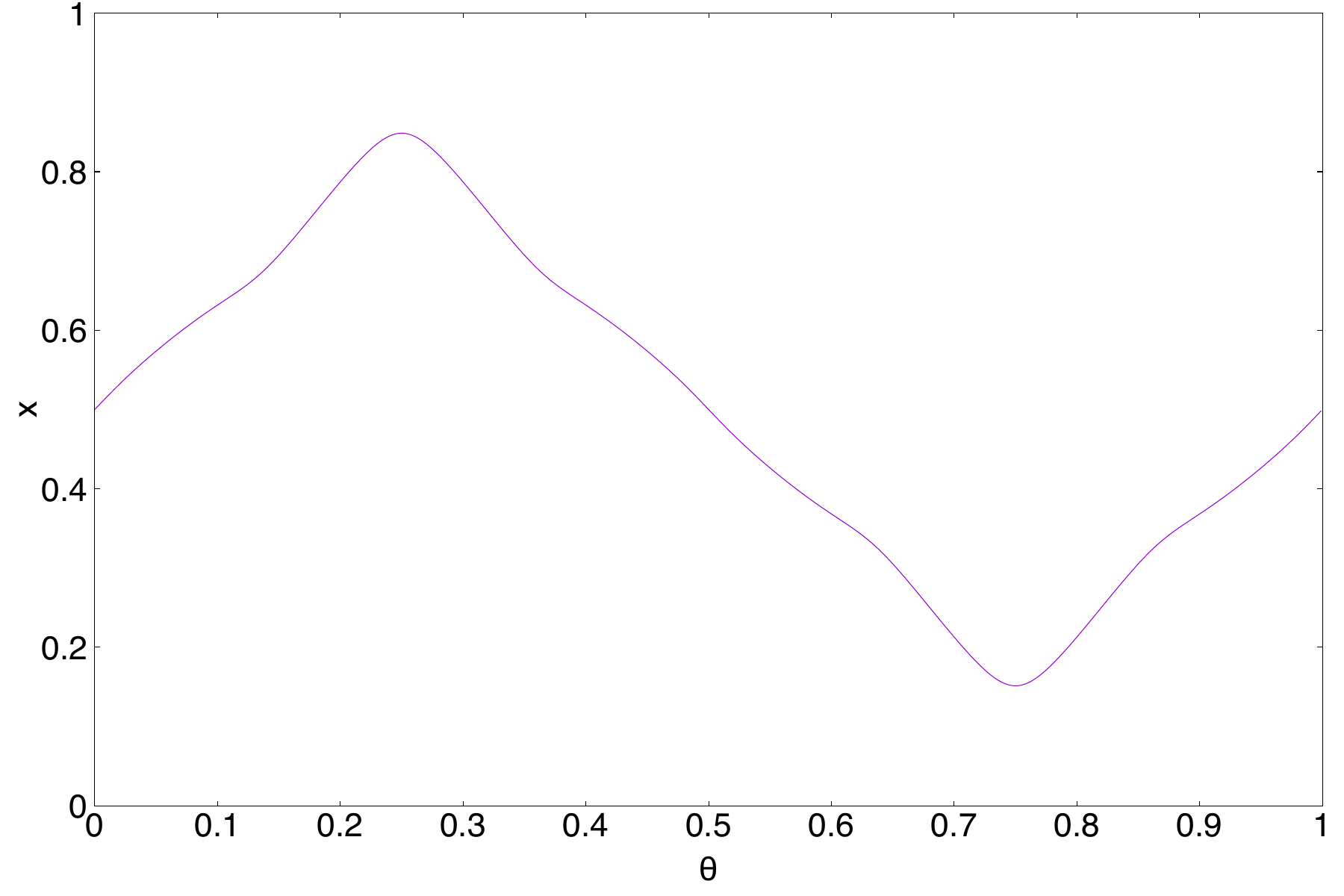}
        \caption{x-projection of the torus, $\epsilon = 1.2$.}
        \label{fig:1e}
    \end{subfigure}
    \hspace{0.5cm}
    \begin{subfigure}{0.45\textwidth}
        \centering
        \includegraphics[width=.89\textwidth]{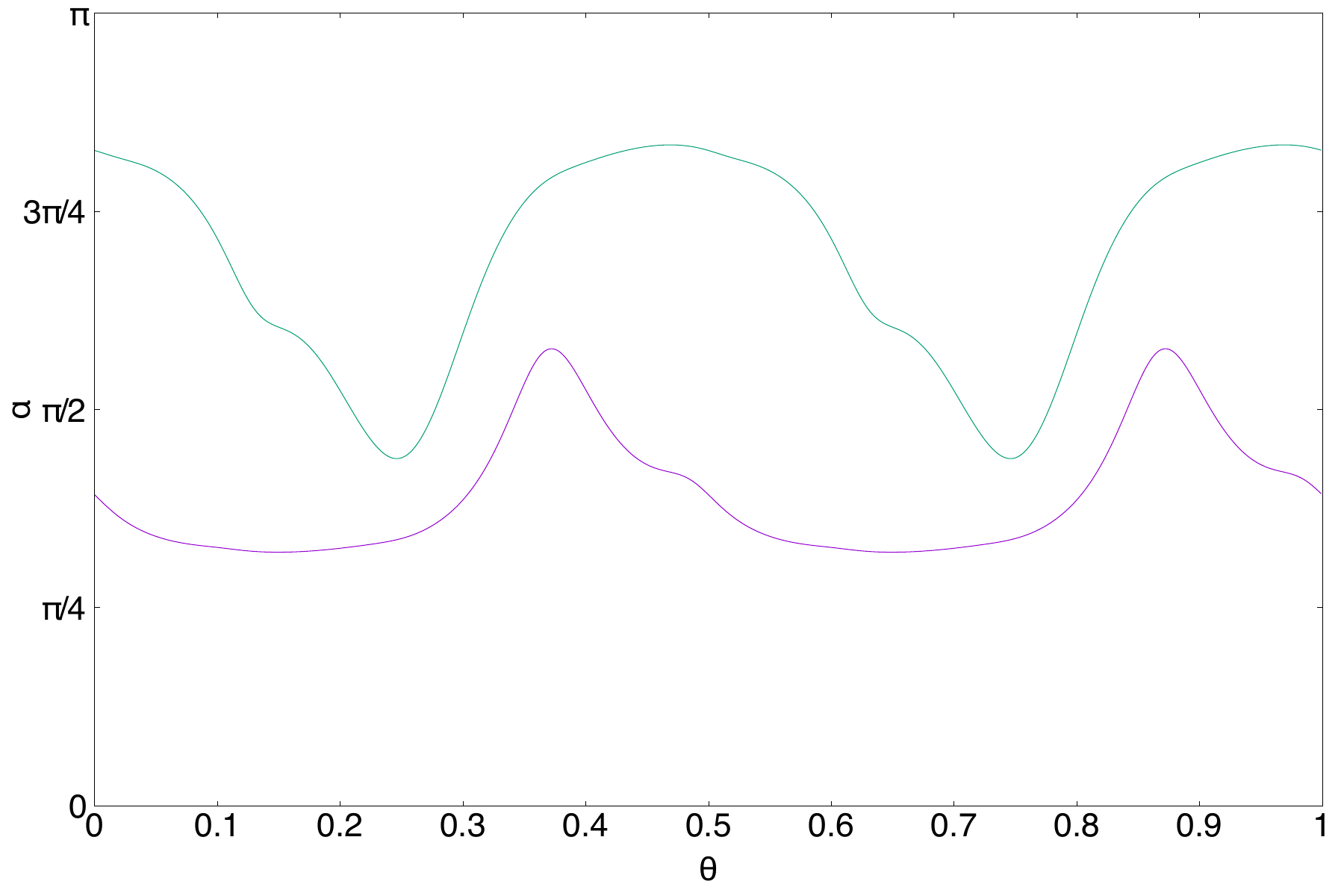}
        \caption{Invariant subbundles, $\epsilon = 1.2$.}
        \label{fig:1f}
    \end{subfigure}

    \medskip

    \begin{subfigure}{0.45\textwidth}
        \centering
        \includegraphics[width=.89\textwidth]{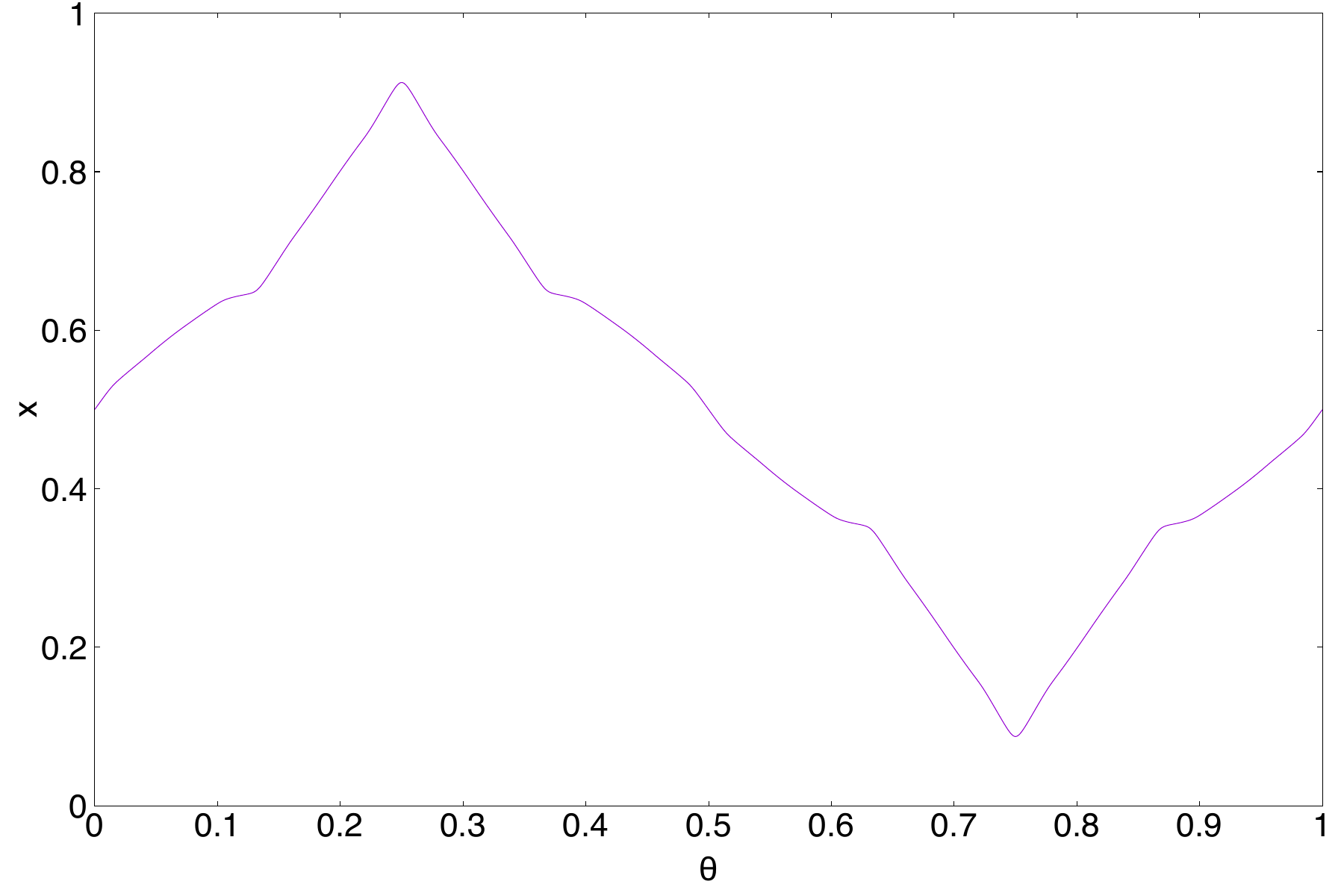}
        \caption{x-projection of the torus, $\epsilon = 1.2342$.}
        \label{fig:1g}
    \end{subfigure}
    \hspace{0.5cm}
    \begin{subfigure}{0.45\textwidth}
        \centering
        \includegraphics[width=.89\textwidth]{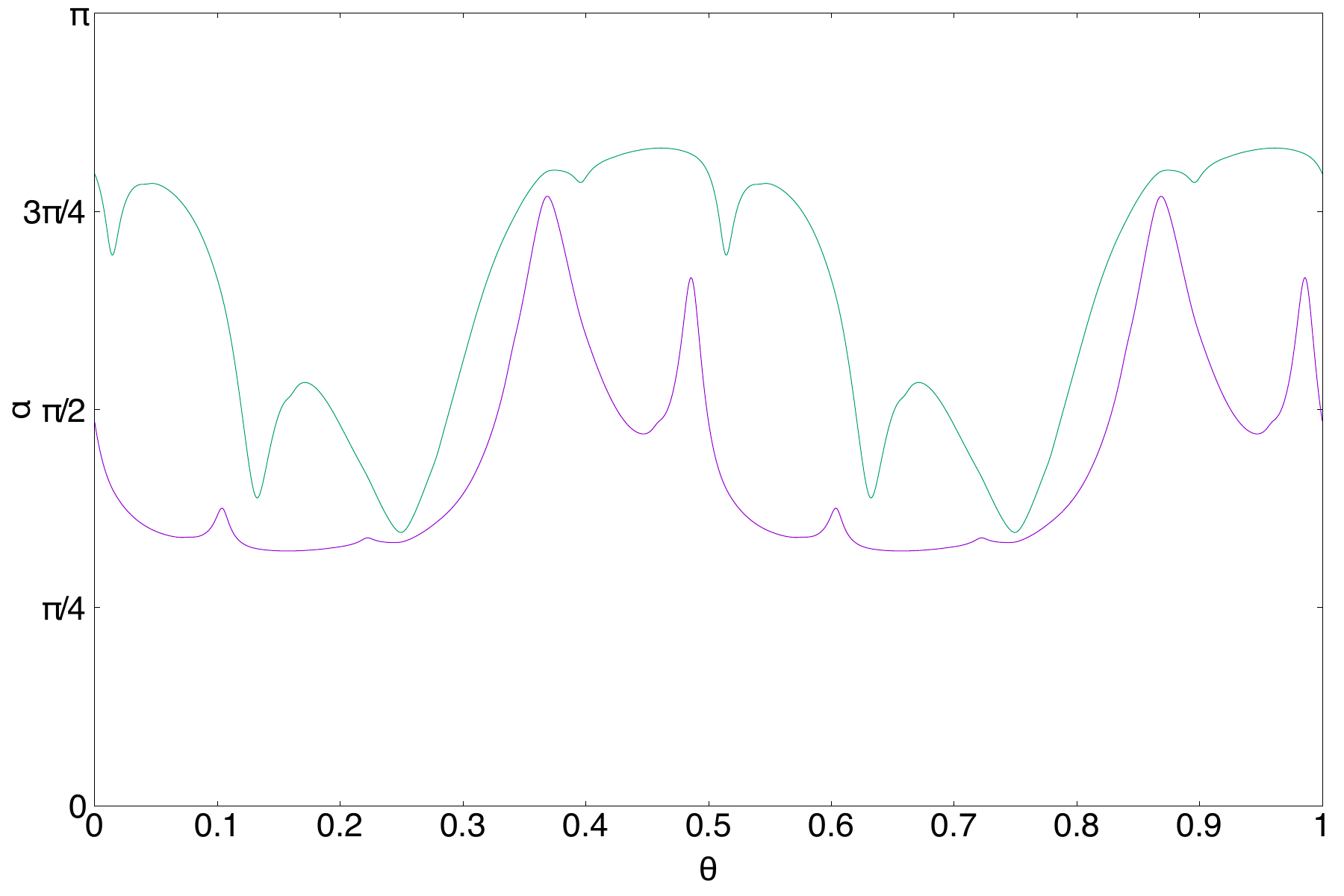}
        \caption{Invariant subbundles, $\epsilon = 1.2342$.}
        \label{fig:1h}
    \end{subfigure}
        \captionsetup{width=\textwidth}
    \caption{Invariant tori (left) and their projectivized invariant subbundles (right) for several values of $\epsilon$ for the standard map \eqref{eq:stdMap}. On the left is displayed the x-projection of the invariant tori and on the right the angle ($\alpha$) between the stable and unstable subbundles and the semiaxis $x>0$.}
    \label{fig:tori}
\end{figure}

\subsection{Validation Examples}

As mentioned in \secref{sec::methodology}, our inputs for the validation are $K_0$, $P_1$, $P_2$, and $\Lambda$. These inputs have been computed with double precision and therefore the validation will be restrained to the maximum precision of double types. It is also worth noting that when we look at the Fourier coefficients of the inputs, we see that after the initial decay of the coefficients, the values get very small. However, when computing Fourier norms, all those values are being multiplied by big exponentials (depending also on $N$), which can amplify that noise and make it significant. That is why it may be a good idea to cut those noise values off and keep only a reasonable neighborhood of the first coefficients. This will of course shrink the grid size but ensure our bounds are not dirtied up by unnecessary noise.

The first term to find then, depending on such inputs, is $\rho$, which, after a choice of $\hatrho$, allows us to compute $C_N(\rho, \hat{\rho})$. We can find a good upper bound estimate for which we can take the value of $\rho$. Notice how the
modulus of the Fourier coefficients of our input torus $K_0$ decays at an exponential rate of $2 \pi N \rho^*$ (times a constant). Therefore, we can perform an exponential regression on such coefficients to find an upper bound $\rho^*$ of the $\rho$ we can use for our validation. In Figure \ref{fig:Kfit}, we can see how the slope $\rho^*$ of the black lines fit the coefficient decay. In a similar way, we can see in Figure \ref{fig:Pfit} that the upper bound for fitting the coefficients of the noisier $P_1$ is going to be higher than for $K_0$.

\begin{figure}
    \begin{subfigure}{0.45\textwidth}
        \centering
        \includegraphics[width=\textwidth]{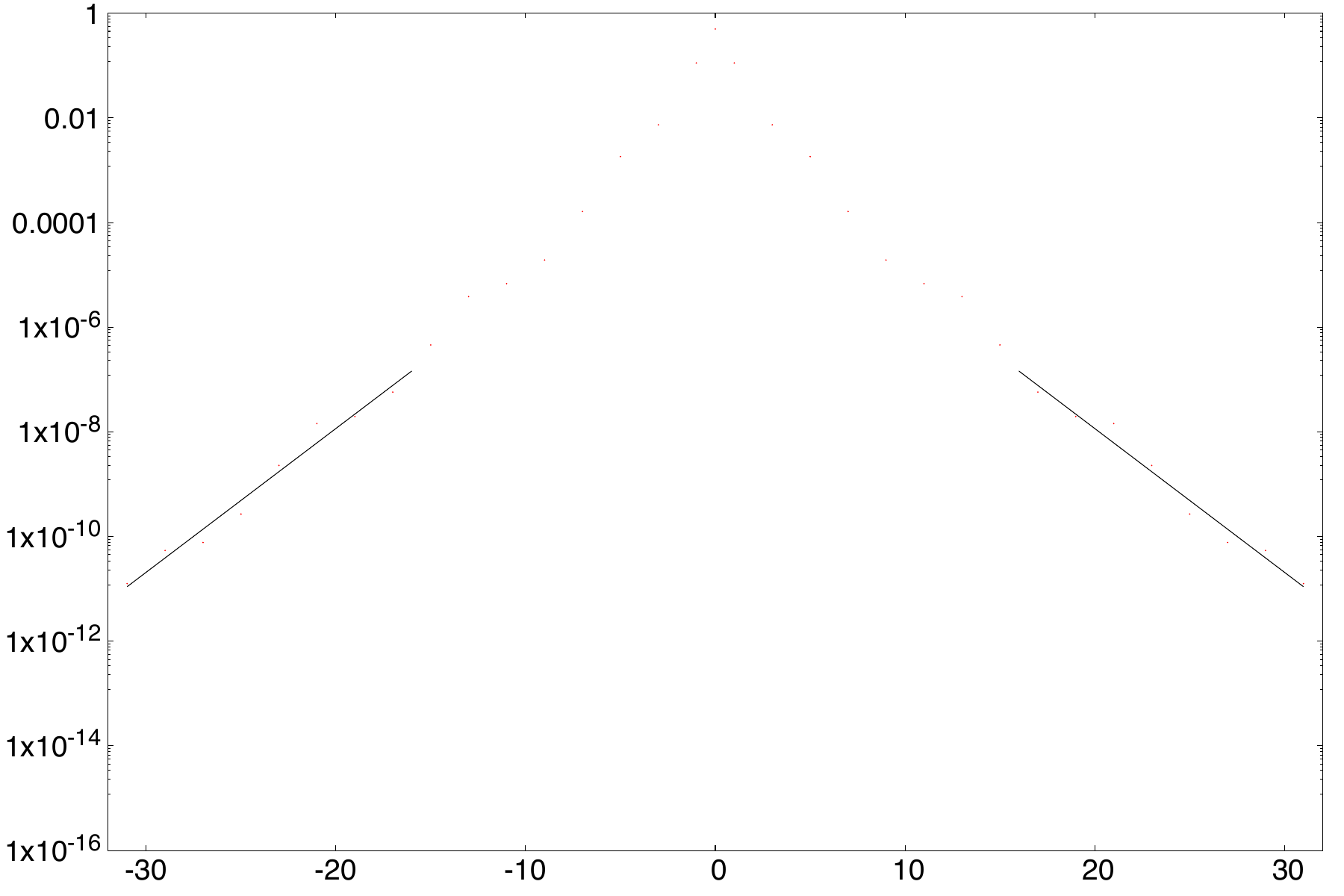}
        \caption{Fourier coefficients and the fitted curve for $\rho^* \approx 0.1$ at $\epsilon = 1$.}
    \end{subfigure}
        \hspace{0.5cm}
    \begin{subfigure}{0.45\textwidth}
        \centering
        \includegraphics[width=\textwidth]{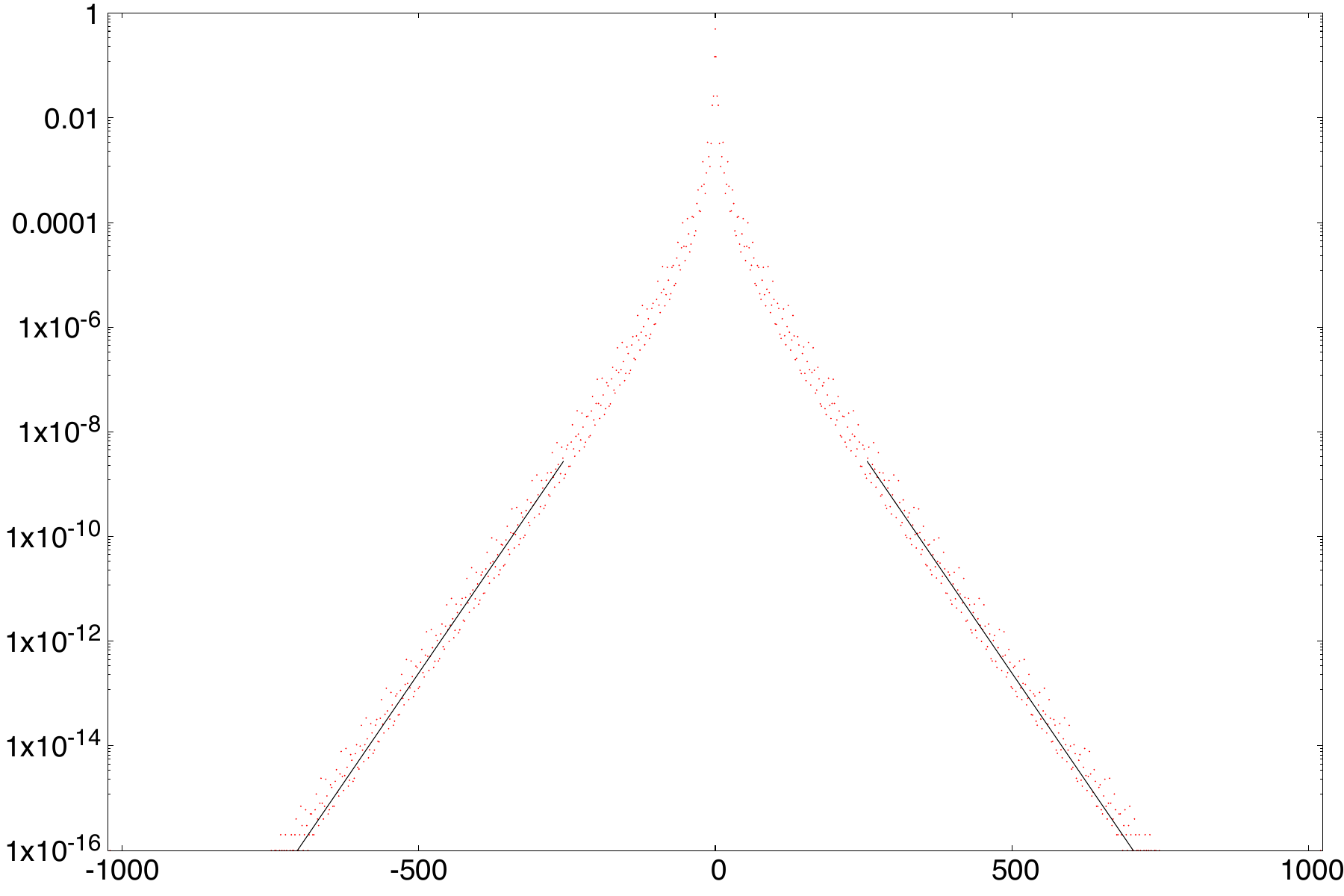}
        \caption{Fourier coefficients and the fitted curve for $\rho^* \approx 0.006$ at $\epsilon = 1.2342$.}
    \end{subfigure}
    \captionsetup{width=\textwidth}
    \caption{Fourier coefficients in a logarithmic scale of the x component of the torus for several values of $\epsilon$ (red) and a regression line of parameter $\rho^*$ fitting the coefficients (black).}
    \label{fig:Kfit}
\end{figure}

\begin{figure}
    \begin{subfigure}{0.45\textwidth}
        \centering
        \includegraphics[width=\textwidth]{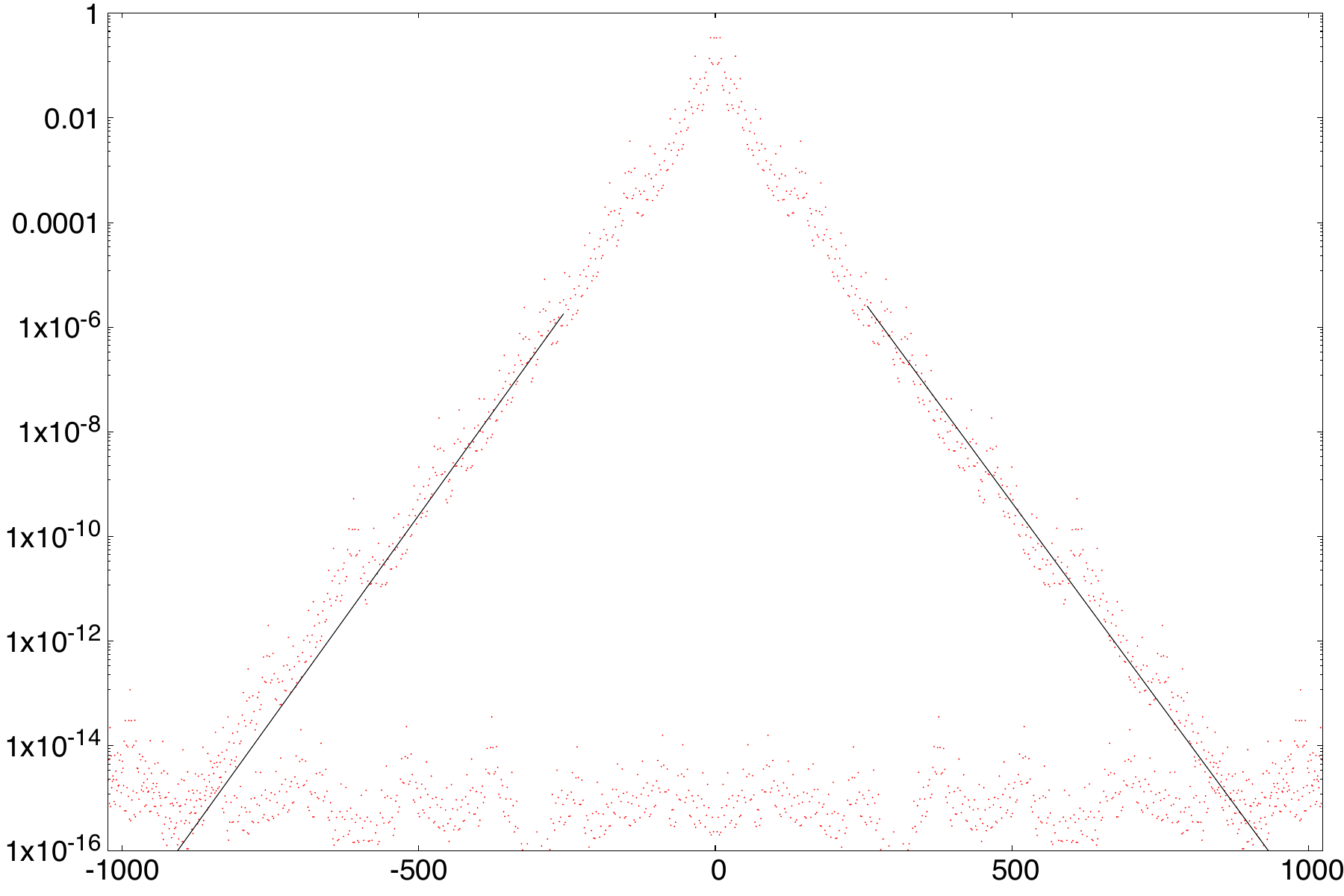}
        \caption{Fourier coefficients and the fitted curve for the stable subbundle.}
    \end{subfigure}
        \hspace{0.5cm}
    \begin{subfigure}{0.45\textwidth}
        \centering
        \includegraphics[width=\textwidth]{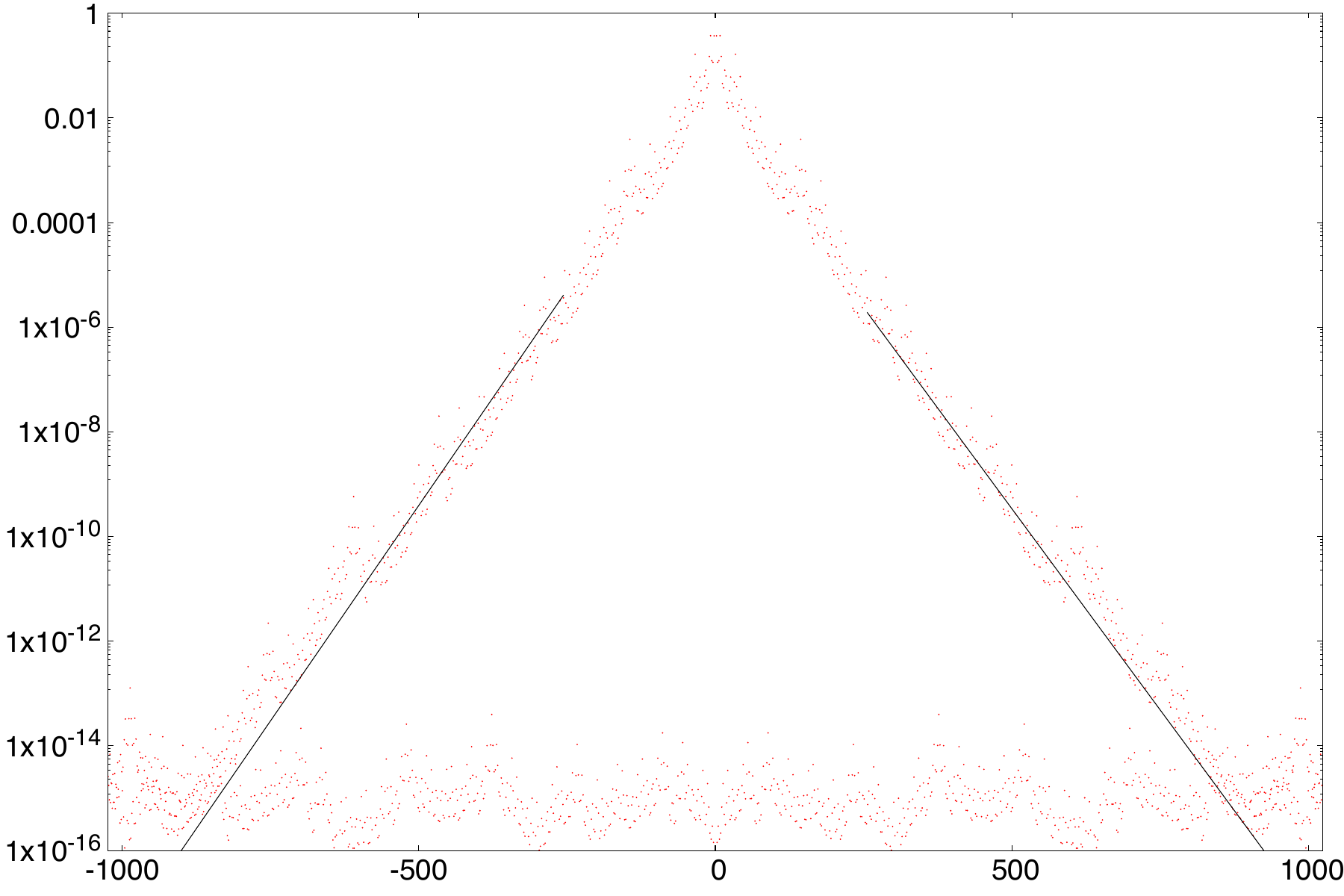}
        \caption{Fourier coefficients and the fitted curve for the unstable subbundle.}
    \end{subfigure}
    \captionsetup{width=\textwidth}
    \caption{Fourier coefficients in a logarithmic scale of the stable subbundle (left) and the unstable subbundle (right) of the torus for $\epsilon = 1.2342$ (red) and a regression line of parameter $\rho^* \approx 0.005$ fitting the coefficients (black).}
    \label{fig:Pfit}
\end{figure}

The first idea here is to choose a $\rho \leq \rho^*$ and a $\hatrho > \rho$ such that $C_N(\rho, \hat\rho)$ becomes as small as possible. Of course we cannot choose any values, as the choice of $\hatrho$ will also have an impact on some objects' norms, just like with $\rho$. Nonetheless, we still require it to be small enough such that the error bounds that are computed afterward are also small enough to guarantee hyperbolicity. Next, we can compute the invariance, reducibility and invertibility error bounds, applying the boxes method as explained in \secref{sec::methodology} when necessary. With such bounds, we can check whether the hyperbolicity condition is satisfied (hence the need for small bounds and values such as $C_N(\rho, \hat{\rho})$), and if so, calculate the hyperbolicity bound $\sigma$. For visualization purposes, we have chosen the same values for the parameters $N$, $\rho$, $\hatrho$, and $R$ to execute the validation for a couple values of $\epsilon$. What can be clearly seen across Tables \ref{tab:eps_0.5} and \ref{tab:eps_1} is that, under the same parameters, all the error bounds increase as the parameter $\epsilon$ increases. This shows how the torus loses invariance and reducibility the closer we get to the critical value of $\epsilon$, which in turn shows also how the torus is losing its hyperbolicity, as $\sigma$ is rapidly increasing, doubling when going just from $\epsilon =0.5$ to $\epsilon=1$. Although not as significant, a similar thing happens to the bound $b(R)$.

\begin{table}[ht]
\centering
\small 
\renewcommand{\arraystretch}{1.2} 
\setlength{\tabcolsep}{3pt} 
\begin{tabular}{|c|c|}
\hline
\textbf{Term} & \textbf{Bound} \\
\hline
$C_N(\rho, \hat{\rho})$ & $9.1889223103814460684374498900346 \times 10^{-8}$ \\
$\NORM{E}$ & $\ep = 1.2828980268004448607526103424455 \times 10^{-7}$ \\
$\NORM{E_{\text{red}}}$ & $\ep_1 = 9.9384120231971379614823261022261 \times 10^{-7}$ \\
$\NORM{E_{\text{inv}}}$ & $\ep_2 = 2.6492816206529570277935091994725 \times 10^{-7}$ \\
$\NORM{(\calM - \calI)^{-1}}$ & $\sigma = 3.5047863969274353212557286917070$ \\
$b(r)$ & $b(R) = 6.2726860980665081720689004716316$ \\
$\NORM{P_1}$ & $1.5338805731759423288761342467093$ \\
$\NORM{P_2}$ & $1.4687982979744499328293125411365$ \\
$\NORM{\Lambda_s}$ & $\lambda_s = 0.357175$ \\
$\NORM{\Lambda_u}$ & $\lambda_u^{-1} = 2.79975$ \\
$\NORM{\Lambda}$ & $\lambda = 0.357175$ \\
\hline
\end{tabular}
\captionsetup{width=\textwidth}
\caption{Validation values for $\kappa=1.3$, $\epsilon = 0.5$, $N = 64$, $\rho = 10^{-2}$, $\hatrho = 10^{-1}$, $R= 1.5 \times 10^{-2}$, $r_-= 5.18586 \times 10^{-7}$, and $r_+ = 1.49999 \times 10^{-2}$.}
\label{tab:eps_0.5}
\end{table}

\begin{table}[ht]
\centering
\small 
\renewcommand{\arraystretch}{1.2} 
\setlength{\tabcolsep}{3pt} 
\begin{tabular}{|c|c|}
\hline
\textbf{Term} & \textbf{Bound} \\
\hline
$C_N(\rho, \hat{\rho})$ & $9.1889223103814460684374498900346 \times 10^{-8}$ \\
$\NORM{E}$ & $\ep = 6.9886143393896989222160867344917 \times 10^{-7}$ \\
$\NORM{E_{\text{red}}}$ & $\ep_1 = 3.2633447245349282716855911598803 \times 10^{-4}$ \\
$\NORM{E_{\text{inv}}}$ & $\ep_2 = 1.0925073901570803841407006747087 \times 10^{-5}$ \\
$\NORM{(\calM - \calI)^{-1}}$ & $\sigma = 7.6699450817858685652216386803508$ \\
$b(r)$ & $b(R) = 8.8552159445544865852315598397991$ \\
$\NORM{P_1}$ & $2.0186640836165030040458099350296$ \\
$\NORM{P_2}$ & $2.1000421847920570152253401199118$ \\
$\NORM{\Lambda_s}$ & $\lambda_s = 0.44695$ \\
$\NORM{\Lambda_u}$ & $\lambda_u^{-1} = 2.23739$ \\
$\NORM{\Lambda}$ & $\lambda = 0.44695$ \\
\hline
\end{tabular}
\captionsetup{width=\textwidth}
\caption{Validation values for $\kappa=1.3$, $\epsilon = 1$, $N = 64$, $\rho = 10^{-2}$, $\hatrho = 10^{-1}$, $R= 1.5 \times 10^{-2}$, $r_-= 5.40288 \times 10^{-6}$, and $r_+ = 1.47234 \times 10^{-2}$.}
\label{tab:eps_1}
\end{table}

It is worth mentioning that although the validation is successful for such cases, the value for $r_-$ gets  bigger the bigger the $\epsilon$. This is because, since our estimates are getting worse, the capacity to ensure uniqueness of a truly invariant torus gets impaired, and therefore the radius of the tube centered in the invariant torus for which we can ensure uniqueness gets smaller (which relates to $r_-$ getting bigger, as the radii at which the Banach fixed-point theorem ensures uniqueness).

What is also interesting is the increasing of the values of the norms of $P_1$ and $P_2$ when parameter $\epsilon$ increases. Recall that, in this case, the first column of $P_1$ represents the direction of the stable subbundle, and the second column represents the direction of the unstable subbundle. This increase means that, although the validated torus is still hyperbolic for such values of $\epsilon$, as the parameter gets closer to the critical value, the fibers show more pronounced changes and get very close at some points (see Figures \ref{fig:1b}, \ref{fig:1d}, \ref{fig:1f}, and \ref{fig:1h}). Therefore, in order to bring each point of the torus of each fiber to the same coordinate system, we will require a less regular change of variables, represented by the matrix $P_1$, which implies that the norms of $P_1$ and $P_2$ get progressively larger. This can become problematic as some of the bounds (such as the reducibility and invertibility error bounds) depend on the norms of $P_1$ and $P_2$. One would be inclined to think that reducing $\hatrho$ would be a solution to this problem, since evaluating the maps $P_1$ and $P_2$ over a narrower band would yield a smaller supremum. However, reducing $\hatrho$ comes with a cost, and that is that $\rho - \hatrho$ becomes smaller and therefore the controlling term $C_N(\rho, \hat{\rho})$ gets bigger. This would still cause the bound to remain big, and that is why the validation gets more intricate the closer we get to the critical value $\epsilon_c$.

\begin{table}[ht]
\centering
\small 
\renewcommand{\arraystretch}{1.2} 
\setlength{\tabcolsep}{3pt} 
\begin{tabular}{|c|c|}
\hline
\textbf{Term} & \textbf{Bound} \\
\hline
$C_N(\rho, \hat{\rho})$ & $9.8707312171782026833836734540706 \times 10^{-8}$ \\
$\NORM{E}$ & $\ep = 2.9326180147973081272937376851241 \times 10^{-7}$ \\
$\NORM{E_{\text{red}}}$ & $\ep_1 = 2.0734245146625048668499036868668 \times 10^{-4}$ \\
$\NORM{E_{\text{inv}}}$ & $\ep_2 = 5.6118904929900122124863469654664 \times 10^{-5}$ \\
$\NORM{(\calM - \calI)^{-1}}$ & $\sigma = 3.6213041171848072135862902799712 \times 10^{2}$ \\
$b(r)$ & $b(R) = 8.4382948067235647008416493688900$ \\
$\NORM{P_1}$ & $1.0471274358468112823731630281721 \times 10$ \\
$\NORM{P_2}$ & $1.1319070929210905461944029023272 \times 10$ \\
$\NORM{\Lambda_s}$ & $\lambda_s = 0.672437$ \\
$\NORM{\Lambda_u}$ & $\lambda_u^{-1} = 1.48713$ \\
$\NORM{\Lambda}$ & $\lambda = 0.672437$ \\
\hline
\end{tabular}
\captionsetup{width=\textwidth}
\caption{Validation values for $\kappa=1.3$, $\epsilon = 1.2342$, $N = 2048$, $\rho = 7 \times 10^{-4}$, $\hatrho = 4 \times 10^{-3}$, $R= 1.5 \times 10^{-2}$, $r_-= 1.33381 \times 10^{-4}$, and $r_+ = 3.27249 \times 10^{-4}$.}
\label{tab:eps_1.2342}
\end{table}

It is at this point that we have to start playing around a bit with $N$, $\rho$, and $\hatrho$ to satisfy the validation conditions, such as reducing $\hatrho$ enough such that the norms of the $P$ matrices are small enough but not too much such that  $C_N(\rho, \hat{\rho})$ gets too big. For that, it is also helpful to increase the amount of Fourier modes $N$ and play around with the other parameters. This can be done in two different ways; first, our input can already be processed in such a way that the grid size is already bigger accounting for this situation, and second, we can manually increase $N$ by increasing the tail of the Fourier coefficients, that is, filling with zeroes. Having a larger $N$ allows for more generous values of $\rho$ and $\hatrho$, which permits a choice of $\rho$ that is closer to the fitting value $\rho^*$.

This behavior can be observed in Table \ref{tab:eps_1.2342}. The parameter values used in Tables \ref{tab:eps_0.5} and \ref{tab:eps_1} will not work anymore for $\epsilon = 1.2342$, so we have to fine tune them. The first thing is to increase the amount of Fourier modes to $2048$, and then find $\rho$ and $\hatrho$ accordingly. With these values, it can be seen that although the validation is successful, the hyperbolicity bound $\sigma$ is substantially bigger than for other values of $\epsilon$, essentially because the norms of the change of variables $P_1$ and $P_2$ increase, while the rates of contraction and expansion remain relatively far from $1$. This is the signature of the bundle collapse leading to the destruction
of a saddle-type invariant torus described in \cite{haroBreakdown,haroTrilogy} (see \cite{CallejaFigueras2012, CanadellHaro2017b, JalnineOsbaldestin2005} for other scenarios).

The code used to perform the proofs is available at \cite{Code} and uses the interval arithmetic package MPFI \cite{mpfi}. The plots have been generated by gnuplot using the outputs of the same code.

\subsection*{Acknowledgements}

The authors are grateful to Jordi-Llu\'is Figueras for fruitful discussions.

\bibliographystyle{plain}
\bibliography{references}

\begin{thebibliography}{10}

\bibitem{agrawal2020universality}
Utkarsh Agrawal, Sarang Gopalakrishnan, and Romain Vasseur.
\newblock Universality and quantum criticality in quasiperiodic spin chains.
\newblock {\em Nature communications}, 11(1):2225, 2020.

\bibitem{van2018continuation}
Jan Bouwe van~den Berg, Maxime Breden, Jean-Philippe Lessard, and Maxime
  Murray.
\newblock Continuation of homoclinic orbits in the suspension bridge equation:
  a computer-assisted proof.
\newblock {\em Journal of Differential Equations}, 264(5):3086--3130, 2018.

\bibitem{CallejaFigueras2012}
Renato Calleja and Jordi-Llu\'{\i}s Figueras.
\newblock Collision of invariant bundles of quasi-periodic attractors in the
  dissipative standard map.
\newblock {\em Chaos}, 22(3):033114, 10, 2012.

\bibitem{CanadellHaro2017b}
Marta Canadell and \`{A}lex Haro.
\newblock Computation of quasi-periodic normally hyperbolic invariant tori:
  algorithms, numerical explorations and mechanisms of breakdown.
\newblock {\em J. Nonlinear Sci.}, 27(6):1829--1868, 2017.

\bibitem{canadellHaro2017}
Marta Canadell and {\`A}lex Haro.
\newblock Computation of quasiperiodic normally hyperbolic invariant tori:
  rigorous results.
\newblock {\em Journal of Nonlinear Science}, 27(6):1869--1904, 2017.

\bibitem{cooley1965algorithm}
James~W Cooley and John~W Tukey.
\newblock An algorithm for the machine calculation of complex fourier series.
\newblock {\em Mathematics of computation}, 19(90):297--301, 1965.

\bibitem{fenichel1974asymptotic}
Neil Fenichel.
\newblock Asymptotic stability with rate conditions.
\newblock {\em Indiana University Mathematics Journal}, 23(12):1109--1137,
  1974.

\bibitem{fenichel1977asymptotic}
Neil Fenichel.
\newblock Asymptotic stability with rate conditions, ii.
\newblock {\em Indiana University Mathematics Journal}, 26(1):81--93, 1977.

\bibitem{feudel1997phase}
Ulrike Feudel, Celso Grebogi, and Edward Ott.
\newblock Phase-locking in quasiperiodically forced systems.
\newblock {\em Physics reports}, 290(1-2):11--25, 1997.

\bibitem{FiguerasHaroLuque2017}
J-Ll Figueras, Alex Haro, and Alejandro Luque.
\newblock Rigorous computer-assisted application of kam theory: a modern
  approach.
\newblock {\em Foundations of Computational Mathematics}, 17:1123--1193, 2017.

\bibitem{FiguerasHaro2012}
Jordi-Llu{\'\i}s Figueras and {\`A}lex Haro.
\newblock Reliable computation of robust response tori on the verge of
  breakdown.
\newblock {\em SIAM Journal on Applied Dynamical Systems}, 11(2):597--628,
  2012.

\bibitem{FiguerasHaroLuque2020-Rotations}
Jordi-Llu\'{\i}s Figueras, Alex Haro, and Alejandro Luque.
\newblock Effective bounds for the measure of rotations.
\newblock {\em Nonlinearity}, 33(2):700--741, 2020.

\bibitem{greene1979method}
John~M Greene.
\newblock A method for determining a stochastic transition.
\newblock {\em Journal of Mathematical Physics}, 20(6):1183--1201, 1979.

\bibitem{mamotreto}
Alex Haro, Marta Canadell, Jordi-Lluis Figueras, Alejandro Luque, and
  Josep-Maria Mondelo.
\newblock The parameterization method for invariant manifolds.
\newblock {\em Applied mathematical sciences}, 195, 2016.

\bibitem{haroBreakdown}
Alex Haro and Rafael de~la Llave.
\newblock Manifolds on the verge of a hyperbolicity breakdown.
\newblock {\em Chaos: An Interdisciplinary Journal of Nonlinear Science},
  16(1), 2006.

\bibitem{DelallaveHaro2006}
Alex Haro and Rafael de~la Llave.
\newblock A parameterization method for the computation of invariant tori and
  their whiskers in quasi-periodic maps: numerical algorithms.
\newblock {\em Discrete and Continuous Dynamical Systems Series B}, 6(6):1261,
  2006.

\bibitem{harodelallave2006parameterization}
Alex Haro and Rafael de~la Llave.
\newblock A parameterization method for the computation of invariant tori and
  their whiskers in quasi-periodic maps: rigorous results.
\newblock {\em Journal of Differential Equations}, 228(2):530--579, 2006.

\bibitem{haroTrilogy}
Alex Haro and Rafael de~La~Llave.
\newblock A parameterization method for the computation of invariant tori and
  their whiskers in quasi-periodic maps: explorations and mechanisms for the
  breakdown of hyperbolicity.
\newblock {\em SIAM Journal on Applied Dynamical Systems}, 6(1):142, 2007.

\bibitem{hirschPS}
M.~W. Hirsch, C.~C. Pugh, and M.~Shub.
\newblock {\em Invariant manifolds}, volume Vol. 583 of {\em Lecture Notes in
  Mathematics}.
\newblock Springer-Verlag, Berlin-New York, 1977.

\bibitem{JalnineOsbaldestin2005}
Alexey~Yu. Jalnine and Andrew~H. Osbaldestin.
\newblock Smooth and nonsmooth dependence of {L}yapunov vectors upon the angle
  variable on a torus in the context of torus-doubling transitions in the
  quasiperiodically forced {H}\'{e}non map.
\newblock {\em Phys. Rev. E (3)}, 71(1):016206, 14, 2005.

\bibitem{johnson1980analyticity}
Russell~A Johnson.
\newblock Analyticity of spectral subbundles.
\newblock {\em Journal of differential equations}, 35(3):366--387, 1980.

\bibitem{mpfi}
Lester~H. Lange and Paul Zimmermann.
\newblock {MPFI}: A {M}ultiple {P}recision {F}loating-point {R}eliable
  {A}rithmetic {P}ackage with {C}orrect {R}ounding.
\newblock \url{https://perso.ens-lyon.fr/nathalie.revol/software.html}, 2006.

\bibitem{lessard2017computer}
Jean-Philippe Lessard and Jason~D Mireles~James.
\newblock Computer assisted fourier analysis in sequence spaces of varying
  regularity.
\newblock {\em SIAM Journal on Mathematical Analysis}, 49(1):530--561, 2017.

\bibitem{linroth2023computer-arxiv}
Victor Linroth.
\newblock Computer assisted proofs for hyperbolic quasi-periodic invariant tori
  in dissipative twist maps, 2023.

\bibitem{mane1978persistent}
Ricardo Man{\'e}.
\newblock Persistent manifolds are normally hyperbolic.
\newblock {\em Transactions of the American Mathematical Society},
  246:261--283, 1978.

\bibitem{prasad2001strange}
Awadhesh Prasad, Surendra~Singh Negi, and Ramakrishna Ramaswamy.
\newblock Strange nonchaotic attractors.
\newblock {\em International Journal of bifurcation and Chaos},
  11(02):291--309, 2001.

\bibitem{Code}
Eric Sandin~Vidal and Alex Haro.
\newblock Fourier methods for caps of invariant tori.
\newblock
  \url{https://github.com/esandivi9/Fourier-Methods-for-CAPS-of-Invariant-Tori.git},
  2024.

\end{thebibliography}

\end{document}